\documentclass[11pt]{article}

\usepackage[margin=1in]{geometry}
\usepackage[T1]{fontenc}
\usepackage{lmodern}
\usepackage{microtype}
\usepackage{amsmath,amssymb,amsthm,mathtools,bm}
\usepackage{mathrsfs}
\usepackage{booktabs,array}
\usepackage{graphicx}
\usepackage{subcaption}
\usepackage{placeins}
\usepackage{xcolor}
\usepackage{enumitem}
\usepackage[colorlinks=true,linkcolor=blue!45!black,
            citecolor=blue!45!black,urlcolor=blue!45!black]{hyperref}
\usepackage[nameinlink,noabbrev]{cleveref}
\usepackage{tikz}

\allowdisplaybreaks
\numberwithin{equation}{section}
\setlist[itemize]{topsep=0.35em,itemsep=0.2em}
\setlist[enumerate]{topsep=0.35em,itemsep=0.2em}

\newtheorem{theorem}{Theorem}[section]
\newtheorem{lemma}[theorem]{Lemma}
\newtheorem{proposition}[theorem]{Proposition}
\newtheorem{corollary}[theorem]{Corollary}
\theoremstyle{remark}
\newtheorem{remark}[theorem]{Remark}
\newtheorem{example}[theorem]{Example}
\newtheorem{definition}[theorem]{Definition}

\newcommand{\T}{\mathbb T}

\newcommand{\cS}{\mathcal S}

\newcommand{\cF}{\mathcal F}

\newcommand{\ip}[2]{\left\langle #1,#2\right\rangle}
\newcommand{\norm}[1]{\left\lVert #1\right\rVert}
\newcommand{\abs}[1]{\left\lvert #1\right\rvert}

\title{\textbf{The sharp CFL condition of the piecewise constant sparse grid discontinuous Galerkin method for high-dimensional transport equations}}
\author{Juntao Huang\thanks{Department of Mathematical Sciences, University of Delaware, Newark, DE 19716, USA. Email: \href{mailto:huangjt@udel.edu}{huangjt@udel.edu}}}
\date{\today}

\begin{document}

\maketitle
\vspace{-2.2em}

\begin{abstract}
    We establish the sharp CFL condition for the piecewise constant sparse grid discontinuous Galerkin (DG) method with forward Euler time stepping, applied to transport equations with constant coefficients on periodic domains in arbitrary dimensions. For the transport velocity $\boldsymbol c=(c_1,\ldots,c_d)$ and a uniform mesh of size $h$, we prove that the scheme is $L^2$ stable if and only if $\Delta t \leq {h}/{\max_{1\leq \ell\leq d}|c_\ell|}$, whereas the corresponding full grid upwind scheme is well-known to require $\Delta t\leq h/\sum_{\ell=1}^d |c_\ell|$. The sparse grid discretization therefore enlarges the admissible time step by a factor of ${(\sum_{\ell=1}^d |c_\ell|)}/{(\max_{1\leq \ell\leq d}|c_\ell|)}$, which lies between $1$ and $d$ and reaches $d$ for isotropic transport. The proof of sufficiency relies on projection leakage identities for the multilevel Haar decomposition, which allow the mixed directional terms in the energy estimate to be absorbed by the energy discarded by the sparse grid projection. The proof of sharpness follows from alternating modes in one dimension at the finest level. As a by-product, we obtain explicit formulas for the $L^2$ operator norm and the spectral radius of the amplification operator. For spaces over general downward closed index sets, we derive an explicit sufficient CFL condition and a geometric criterion for its sharpness. Numerical experiments in two and four dimensions confirm the theoretical results.
\end{abstract}


\section{Introduction}
\label{sec:introduction}

In this paper, we study the CFL condition for the piecewise constant sparse grid discontinuous Galerkin (DG) method, applied to the transport equation and periodic boundary conditions in arbitrary dimensions $d\ge 1$:
\begin{equation}\label{eq:intro-transport}
    u_t+\sum_{\ell=1}^d c_\ell u_{x_\ell}=0.
\end{equation}

The DG method was proposed by Reed and Hill \cite{reed1973triangularmesh} in the framework of neutron transport. It was then developed for conservation laws by Cockburn et al. in a series of papers \cite{cockburn1991runge1,cockburn1989tvb2,cockburn1989tvb3,cockburn1990runge4,cockburn1998runge5}, which use piecewise polynomials in space and explicit total variation diminishing (TVD) Runge-Kutta (RK) discretization in time \cite{shu1988efficient,gottlieb2001strong}. The method offers several advantages, including flexibility in handling complex geometries, $hp$-adaptivity, natural preservation of local conservation, and high efficiency for parallelization. For the detailed description of this method, we refer readers to the lecture notes \cite{cockburn2006introduction} and the review paper \cite{cockburn2001runge}.

In high dimensions, however, discretizations based on tensor product grids suffer from the curse of dimensionality: with mesh size $h$ in each direction, the number of degrees of freedom grows as $O(h^{-d})$. 
Sparse grid methods alleviate this growth by retaining only a selected subset of hierarchical tensor product components \cite{griebel1990combination,bungartz2004sparse,lastdrager2001sparse,schwab2008sparse}, which reduces the number of degrees of freedom to $O(h^{-1}\abs{\log_2 h}^{d-1})$. 
Sparse grid DG method was developed for high-dimensional elliptic equations \cite{wang2016sparse} and transport equations \cite{guo2016sparse,guo2017adaptive}, and has been extended to nonlinear conservation laws \cite{huang2020adaptive}, wave equations \cite{huang2020adaptiveIPDG}, nonlinear dispersive equations \cite{tao2022adaptive,huang2022class}, the Vlasov-Maxwell system \cite{tao2019vlasovmaxwell}, and Hamilton-Jacobi equations \cite{guo2021adaptive}. See \cite{huang2024adaptive} for a recent review on the sparse grid DG method.

For explicit time integrations, the computational cost also depends on the admissible time step. On a uniform full grid, the piecewise constant upwind DG method is equivalent to the classical first-order upwind scheme \cite{courant1928partiellen,gustafsson1995time}. With forward Euler time stepping, it is $L^2$-contractive if and only if
\begin{equation}\label{eq:intro-full-grid-cfl}
    \Delta t\leq\frac{h}{\sum_{\ell=1}^d\abs{c_\ell}},
\end{equation}
where $h$ is the mesh size in each direction. The fully discrete $L^2$ stability of higher-order Runge-Kutta DG methods on full grids has also been studied extensively in \cite{zhang2010stability,sun2017stability,sun2019strong,xu20192}. 
Under the sparse grid discretization, however, the situation is more subtle: the sparse grid space is constructed from wavelets at multiple hierarchical levels and the sharp CFL condition can no longer be derived from these full grid analyses. To our knowledge, no theoretical analysis of the stability or the CFL condition for a fully discrete sparse grid DG scheme is available in the literature.

Numerical investigations have shown that the sparse grid method can lead to a larger time step \cite{tao2019sparse}. For two-dimensional isotropic transport (i.e. $u_t + u_x + u_y = 0$), their eigenvalue calculations indicated that the CFL number of the sparse grid DG schemes were approximately twice those of the corresponding full grid schemes. These  numerical observations motivate a rigorous analysis of the CFL condition and its dependence on the dimension and transport velocity.

The purpose of this paper is to address this question. Our contributions are threefold.

First, we prove that the piecewise constant sparse grid DG method with forward Euler time stepping is $L^2$ stable if and only if
\begin{equation}\label{eq:intro-sparse-grid-cfl}
    \Delta t\leq\frac{h}{\max_{1\leq\ell\leq d}\abs{c_\ell}}.
\end{equation}
Thus, the CFL condition is governed by the largest speed, rather than by the sum of the speeds in \eqref{eq:intro-full-grid-cfl}. The admissible time step is enlarged by a factor of ${(\sum_{\ell=1}^d\abs{c_\ell})}/{(\max_{1\leq\ell\leq d}\abs{c_\ell})}$, which lies between $1$ and $d$ and equals $d$ for isotropic transport. The analysis further yields exact formulas for the $L^2$ norm and the spectral radius of the amplification operator.

Second, we extend the analysis to hierarchical spaces built on arbitrary downward closed index sets, which include anisotropic sparse grid methods \cite{bungartz2004sparse,griebel2000optimized} and adaptive sparse grid methods \cite{guo2017adaptive}. We identify an explicit combinatorial quantity of the index set $\Lambda$, which yields a sufficient CFL condition, together with a geometric criterion on $\Lambda$ under which this condition is sharp. Both the full grid space and the sparse grid space satisfy the criterion, so the classical condition \eqref{eq:intro-full-grid-cfl} and the new condition \eqref{eq:intro-sparse-grid-cfl} arise as the two extreme cases of a single formula. An L-shaped index set further shows that the condition can be strictly sufficient in general.

Third, the proofs are based on projection leakage identities for the multilevel Haar decomposition, which quantify exactly how a difference operator in a fine grid splits into the parts retained and discarded by a coarse grid projection. With these identities, we show that the energy discarded by the hierarchical projection dominates the mixed directional terms in the energy estimate. The exactness of these identities is what makes the sharp characterization of the CFL condition possible. We expect that these tools will be useful in the fully discrete stability analysis of other multilevel and multiresolution discretizations.

The rest of the paper is organized as follows. We begin with a brief introduction to the Haar wavelets and the piecewise constant sparse grid DG scheme in Section~\ref{sec:setting}. Section~\ref{sec:proof2d} develops the projection leakage identities and proves the sharp CFL condition in two dimensions. Section~\ref{sec:proof-arbitrary-d} then extends the argument to arbitrary spatial dimensions. Section~\ref{sec:downward-closed} generalizes the analysis to arbitrary downward closed index sets. Numerical experiments in Section~\ref{sec:numerics} validate the theoretical results. Section~\ref{sec:conclusion} concludes the paper.

\section{Preliminaries}
\label{sec:setting}

In this section, we review the Haar wavelets \cite{mallat1999wavelet} and the piecewise constant sparse grid DG method \cite{wang2016sparse,guo2016sparse} for solving the transport equation. Throughout the paper, all function spaces and inner products are real. For a subspace $X$, we use $P_X$ to denote the $L^2$ orthogonal projection onto $X$.

\subsection{Haar wavelets}

Let $\T=[0,1)$ be the one-dimensional periodic domain. For a mesh level $n\geq0$, let $\Omega_n := \{I_{n,j}\}_{j=0}^{2^n-1}$ be the uniform mesh consisting of $2^n$ cells, where
\[    
    I_{n,j}
    =
    [j h_n,(j+1)h_n),
    \qquad
    0\leq j<2^n,
\]
with $h_n=2^{-n}$. Define the piecewise constant polynomial space on level $n$ by
\begin{equation*}
    V_n
    :=
    \left\{
        v\in L^2(\T):
        v\big|_{I_{n,j}}\in\mathbb P^0(I_{n,j})
        \text{ for }0\leq j<2^n
    \right\},
\end{equation*}
where $\mathbb{P}^0(I_{n,j})$ denotes the space of piecewise constant functions on $I_{n,j}$.
The spaces are nested:
\[
    V_0\subset V_1\subset\cdots\subset V_n\subset\cdots.
\]
For $n\geq1$, let $W_n$ be the $L^2$ orthogonal complement of
$V_{n-1}$ in $V_n$:
\begin{equation}\label{eq:W-level}
    V_n=V_{n-1}\oplus W_n,
    \qquad
    W_n\perp V_{n-1}.
\end{equation}
For notational convenience, set $W_0:=V_0$. Then
\begin{equation}\label{eq:haar-decomposition}
    V_n
    =
    \bigoplus_{p=0}^{n}W_p.
\end{equation}
For $n\geq1$, $W_n$ is spanned by the Haar wavelets \cite{mallat1999wavelet}.

On $\T^2$, let $V_n^x,W_n^x$ and $V_n^y,W_n^y$ denote
the corresponding spaces in the $x$- and $y$-variables. The standard
full grid space at level $N$ is
\begin{equation}\label{eq:full-space}
    \cF_N := V_N^x\otimes V_N^y = \bigoplus_{0\leq p,q\leq N} W_p^x\otimes W_q^y.
\end{equation}
The sparse grid space is \cite{wang2016sparse}
\begin{equation}\label{eq:sparse-space}
    \cS_N := \bigoplus_{0\leq p+q\leq N} W_p^x\otimes W_q^y.
\end{equation}

\subsection{Piecewise constant sparse grid DG scheme}

For $v\in V_n$, let $v_j$ denote its value on $I_{n,j}$ for $0\le j < 2^n$. Define the one dimensional periodic backward difference operator $D_n:V_n\to V_n$ by
\begin{equation}\label{eq:D-level-definition}
    (D_n v)_j = v_j-v_{j-1},
\end{equation}
where $j-1$ is understood modulo $2^n$.

On $\cF_N$, define the full grid difference operators
$A_x:\cF_N\to\cF_N$ and $A_y:\cF_N\to\cF_N$ by
\begin{equation}\label{eq:Ax-Ay}
    A_x
    :=
    D_N^x\otimes I_y,
    \qquad
    A_y
    :=
    I_x\otimes D_N^y.
\end{equation}
Here $D_N^x$ and $D_N^y$ are the one-dimensional difference operators \eqref{eq:D-level-definition} in the $x$- and $y$-directions, respectively, while $I_x$
and $I_y$ are the corresponding identity operators.

Consider the transport equation with constant coefficient and periodic boundary conditions:
\begin{equation}\label{eq:transport}
    u_t+c_1u_x+c_2u_y=0,
\end{equation}
with $\boldsymbol c=(c_1,c_2)\in\mathbb{R}^2$. Throughout the analysis, we assume $c_1, c_2\geq0$ without loss of generality; the negative velocity component can be handled in the similar way. Let $h:=2^{-N}$ and let $\Delta t\geq0$ be the time step size. The full grid piecewise constant upwind DG forward Euler scheme is
\begin{equation}\label{eq:full-grid-update}
    u^{n+1}
    =
    u^n
    -\frac{c_1\Delta t}{h}A_xu^n
    -\frac{c_2\Delta t}{h}A_yu^n.
\end{equation}
Testing the full grid DG bilinear form only against $\cS_N$ is equivalent to applying the orthogonal projection onto the sparse grid space \cite{guo2016sparse}. Therefore, the corresponding sparse grid scheme is
\begin{equation}\label{eq:sparse-update}
    u^{n+1} = u^n - \frac{c_1\Delta t}{h}P_{\cS_N}A_xu^n - \frac{c_2\Delta t}{h}P_{\cS_N}A_yu^n.
\end{equation}
where $P_{\cS_N}$ is the $L^2$ projection from $\cF_N$ to $\cS_N$.

\section{Stability analysis of the sparse grid scheme in two dimensions}
\label{sec:proof2d}

It is well-known that the sharp CFL condition for the full grid scheme \eqref{eq:full-grid-update} is
\[
    \Delta t
    \leq
    \frac{h}{c_1+c_2}.
\]
In this section, we will prove that the sharp CFL condition for the
sparse grid scheme \eqref{eq:sparse-update} is
\begin{equation}\label{eq:observed-square}
    \Delta t
    \leq
    \frac{h}{\max(c_1,c_2)}.
\end{equation}

\subsection{One-dimensional projection leakage identities}\label{subsec:1d-compression}

We first establish several identities for the periodic difference operator and the coarse grid projection.
\begin{lemma}\label{lem:D-energy}
For any $v\in V_n$,
\begin{equation}\label{eq:D-energy}
    2\ip{v}{D_n v}
    =
    \norm{D_n v}^2.
\end{equation}
\end{lemma}

\begin{proof}
Let $v_j$ be the value of $v$ on $I_{n,j}$. With periodic indexing,
\begin{equation*}
    \norm{D_n v}^2
    =
    h\sum_j(v_j-v_{j-1})^2
    =
    h\sum_j
    \left(v_j^2+v_{j-1}^2-2v_jv_{j-1}\right)
    =
    2h\sum_jv_j^2-2h\sum_jv_jv_{j-1}
    =
    2\ip{v}{D_n v}.
\end{equation*}
\end{proof}

\begin{lemma}[One-level coarse grid projection identity]
\label{lem:single-level}
For $n\geq1$ and any $v\in V_{n-1}$,
\begin{equation}\label{eq:single-level-compression}
    P_{V_{n-1}}D_n v
    =
    \frac12D_{n-1}v,
\end{equation}
where $v$ on the left-hand side is viewed as a function in $V_n$. 
\end{lemma}

\begin{proof}
Let $v_j$ be the value of $v$ on the coarse cell $I_{n-1,j}$, and let $\widetilde v\in V_n$ denote its embedding into the fine space $V_n$. Then
\[
    \widetilde v_{2j}
    =
    \widetilde v_{2j+1}
    =
    v_j.
\]
Consequently,
\[
    (D_n\widetilde v)_{2j}
    =
    \widetilde v_{2j}-\widetilde v_{2j-1}
    =
    v_j-v_{j-1},
    \qquad
    (D_n\widetilde v)_{2j+1}
    =
    \widetilde v_{2j+1}-\widetilde v_{2j}
    =
    0.
\]
Projection onto $V_{n-1}$ averages the values on the two children
of each coarse cell. Therefore,
\begin{align*}
    \left(P_{V_{n-1}}D_n\widetilde v\right)_j
    =
    \frac12
    \left(
        (D_n\widetilde v)_{2j}
        +(D_n\widetilde v)_{2j+1}
    \right)
    =
    \frac12(v_j-v_{j-1})
    =
    \frac12(D_{n-1}v)_j.
\end{align*}
\end{proof}

\begin{lemma}[Multi-level coarse grid projection identity]
\label{lem:multilevel}
Let $0\leq m\leq N$. For any $v\in V_m$,
\begin{equation}\label{eq:multilevel-compression}
    P_{V_m}D_Nv
    =
    2^{m-N}D_mv.
\end{equation}
\end{lemma}

\begin{proof}
For $n=m+1,\ldots,N$, nestedness gives
\[
    V_m\subset V_{n-1},
    \qquad
    P_{V_m}P_{V_{n-1}}
    =
    P_{V_m}.
\]
Applying Lemma~\ref{lem:single-level},
\begin{align*}
    P_{V_m}D_n v
    =
    P_{V_m}P_{V_{n-1}}D_n v
    =
    \frac12P_{V_m}D_{n-1}v.
\end{align*}
Repeating this identity for $n=N,N-1,\ldots,m+1$ yields
\[
    P_{V_m}D_Nv
    =
    2^{m-N}P_{V_m}D_mv.
\]
Since $D_mv\in V_m$, we have $P_{V_m}D_mv=D_mv$.
\end{proof}

We are now ready to prove the main result of this subsection. It quantifies how much energy is discarded by the coarse grid projection, when the fine grid difference operator is applied to a function in the coarse grid space.
\begin{proposition}[Multi-level projection leakage identity]
\label{prop:1d-leakage}
Let $v\in V_m\subset V_N$ for $0\leq m\leq N$, and define
\[
    g
    :=
    P_{V_m}D_Nv,
    \qquad
    e
    :=
    (I-P_{V_m})D_Nv.
\]
Then
\begin{equation}\label{eq:1d-leakage}
    \norm{e}^2
    =
    \left(2^{N-m}-1\right)\norm{g}^2.
\end{equation}
\end{proposition}

\begin{proof}
Set $\alpha:=2^{m-N}$. By Lemma~\ref{lem:multilevel},
\[
    g=\alpha D_mv,
    \qquad
    \norm{g}^2
    =
    \alpha^2\norm{D_mv}^2.
\]
Since $e\perp V_m$ and $v\in V_m$, we have $\ip{v}{e}=0$. Applying
Lemma~\ref{lem:D-energy} at levels $N$ and $m$ gives
\begin{align*}
    \norm{D_Nv}^2
    =
    2\ip{v}{D_Nv}
    =
    2\ip{v}{g+e}
    =
    2\alpha\ip{v}{D_mv}
    =
    \alpha\norm{D_mv}^2.
\end{align*}
Because $D_Nv=g+e$ is an orthogonal decomposition,
\begin{align*}
    \norm{e}^2
    =
    \norm{D_Nv}^2-\norm{g}^2
    =
    (\alpha-\alpha^2)\norm{D_mv}^2
    =
    (\alpha^{-1}-1)\norm{g}^2
    =
    \left(2^{N-m}-1\right)\norm{g}^2.
\end{align*}
\end{proof}

\subsection{Two-dimensional projection leakage identities}\label{subsec:2d-leakage}

Now we introduce the two-dimensional projection leakage identities, which quantifies the projection leakage energy of the sparse grid projection. This serves as an important tool for the stability analysis of the sparse grid scheme.

\begin{lemma}\label{lem:fixed-level-projection}
Let $0\leq q\leq N$. For any $\phi\in V_N^x$
and $\psi\in W_q^y$,
\[
    P_{\cS_N}\bigl(\phi\psi\bigr)
    =
    \bigl(P_{V_{N-q}^x}\phi\bigr)\psi.
\]
More generally, for any $f\in V_N^x\otimes W_q^y$,
\[
    P_{\cS_N}f
    =
    \left(P_{V_{N-q}^x}\otimes I_y\right)f.
\]
\end{lemma}

\begin{proof}
The orthogonal decomposition \eqref{eq:sparse-space} gives $P_{\cS_N} = \sum_{p+r\leq N}P_{W_p^x}\otimes P_{W_r^y}$.
For $\psi\in W_q^y$, the orthogonality implies $P_{W_r^y}\psi = \delta_{r,q}\psi$. Thus, for $\phi\in V_N^x$,
\[
\begin{aligned}
    P_{\cS_N}\bigl(\phi(x)\psi(y)\bigr)
    =
    \sum_{p+r\leq N}
    \bigl(P_{W_p^x}\phi\bigr)(x)
    \bigl(P_{W_r^y}\psi\bigr)(y)
    =
    \sum_{p=0}^{N-q}
    \bigl(P_{W_p^x}\phi\bigr)(x)\psi(y)
    =
    \bigl(P_{V_{N-q}^x}\phi\bigr)(x)\psi(y),
\end{aligned}
\]
where we have used the fact that $V_{N-q}^x=\bigoplus_{p=0}^{N-q}W_p^x$ in the last step.

For any $f\in V_N^x\otimes W_q^y$, we can write $f$ as
\[
    f(x,y)
    =
    \sum_{k=1}^{d_q}\phi_k(x)\psi_k(y),
    \qquad \phi_k\in V_N^x, \, \psi_k\in W_q^y.
\]
By linearity and the identity established above,
\[
\begin{aligned}
    P_{\cS_N}f
    =
    \sum_{k=1}^{d_q}
    P_{\cS_N}\bigl(\phi_k(x)\psi_k(y)\bigr)
    =
    \sum_{k=1}^{d_q}
    \bigl(P_{V_{N-q}^x}\phi_k\bigr)(x)\psi_k(y)
    =
    \left(P_{V_{N-q}^x}\otimes I_y\right)f.
\end{aligned}
\]
\end{proof}

\begin{proposition}[Two-dimensional projection leakage identities]
\label{prop:2d-leakage}
For $u\in\cS_N$, define
\[
    X_{pq}
    :=
    P_{W_p^x\otimes W_q^y}A_xu,
    \qquad
    Y_{pq}
    :=
    P_{W_p^x\otimes W_q^y}A_yu,
    \qquad
    0\leq p,q\leq N.
\]
Then
\begin{align}
    \norm{(I-P_{\cS_N})A_xu}^2
    &=
    \sum_{p+q\leq N}
    (2^q-1)\norm{X_{pq}}^2,
    \label{eq:x-leakage}
    \\
    \norm{(I-P_{\cS_N})A_yu}^2
    &=
    \sum_{p+q\leq N}
    (2^p-1)\norm{Y_{pq}}^2.
    \label{eq:y-leakage}
\end{align}
\end{proposition}

\begin{proof}
We only prove \eqref{eq:x-leakage}; the proof of \eqref{eq:y-leakage} follows in the same way by symmetry. The proof consists of three steps, which gradually extend the identity from a single Haar block to the entire sparse grid space.

\underline{Step 1: Functions in a single Haar block.}
Fix $0\leq q\leq N$ and $0\leq p\leq N-q$.
We first consider a separable function
\[
    v(x,y)=\phi(x)\psi(y),
    \qquad
    \phi\in W_p^x,
    \quad
    \psi\in W_q^y.
\]
Since $A_x=D_N^x\otimes I_y$, we have
\[
    A_xv=(D_N^x\phi)\psi.
\]
By Lemma~\ref{lem:fixed-level-projection},
\[
\begin{aligned}
    P_{\cS_N}A_xv
    =
    \left(P_{V_{N-q}^x}D_N^x\phi\right)\psi,
    \qquad
    (I-P_{\cS_N})A_xv
    =
    \left((I_x-P_{V_{N-q}^x})D_N^x\phi\right)\psi.
\end{aligned}
\]
Because $p\leq N-q$, we have $\phi\in W_p^x\subseteq V_{N-q}^x$. Proposition~\ref{prop:1d-leakage}, applied with $m=N-q$,
 gives
\[
    \norm{(I_x-P_{V_{N-q}^x})D_N^x\phi}_{L_x^2}^2
    =
    (2^q-1)
    \norm{P_{V_{N-q}^x}D_N^x\phi}_{L_x^2}^2.
\]
Multiplying both sides by $\norm{\psi}_{L_y^2}^2$ yields
\[
    \norm{(I-P_{\cS_N})A_xv}^2
    =
    (2^q-1)\norm{P_{\cS_N}A_xv}^2.
\]

Next, for any $v\in W_p^x\otimes W_q^y$. Choose an orthonormal basis
$\{\psi_{q,k}\}_{k=1}^{d_q}$ of $W_q^y$,
where $d_q:=\dim W_q^y$, and write
\[
    v(x,y)
    =
    \sum_{k=1}^{d_q}\phi_k(x)\psi_{q,k}(y),
    \qquad
    \phi_k\in W_p^x.
\]
Applying Lemma~\ref{lem:fixed-level-projection} to each term gives
\[
\begin{aligned}
    P_{\cS_N}A_xv
    =
    \sum_{k=1}^{d_q}
    \left(P_{V_{N-q}^x}D_N^x\phi_k\right)\psi_{q,k},
    \qquad
    (I-P_{\cS_N})A_xv
    =
    \sum_{k=1}^{d_q}
    \left((I_x-P_{V_{N-q}^x})D_N^x\phi_k\right)\psi_{q,k}.
\end{aligned}
\]
The orthonormality of the $\psi_{q,k}$ eliminates the cross terms when taking squared norms in $L^2$. Consequently,
\[
\begin{aligned}
    \norm{(I-P_{\cS_N})A_xv}^2
    &=
    \sum_{k=1}^{d_q}
    \norm{(I_x-P_{V_{N-q}^x})D_N^x\phi_k}_{L_x^2}^2
    \\
    &=
    (2^q-1)
    \sum_{k=1}^{d_q}
    \norm{P_{V_{N-q}^x}D_N^x\phi_k}_{L_x^2}^2
    \\
    &=
    (2^q-1)\norm{P_{\cS_N}A_xv}^2.
\end{aligned}
\]
From this step, we observe that the identity \eqref{eq:x-leakage} holds for any $v\in W_p^x\otimes W_q^y$ with $p+q\leq N$ and the factor $(2^q-1)$ only depends on $q$.

\underline{Step 2: Functions with a fixed $y$-level.}
We next extend this identity to
\[
    v\in V_{N-q}^x\otimes W_q^y
    =
    \bigoplus_{p=0}^{N-q}W_p^x\otimes W_q^y.
\]
Then, we write
\[
    v=\sum_{p=0}^{N-q}v_p,
    \qquad
    v_p\in W_p^x\otimes W_q^y.
\]
Using the same orthonormal basis $\{ \psi_{q,k} \}_{k=1}^{d_q}$ of $W_q^y$ for every $p$, we can expand
\[
    v_p(x,y)
    =
    \sum_{k=1}^{d_q}\phi_{p,k}(x)\psi_{q,k}(y),
    \qquad
    \phi_{p,k}\in W_p^x.
\]
Collecting the coefficients of each $\psi_{q,k}$ gives
\[
    v(x,y)
    =
    \sum_{k=1}^{d_q}\Phi_k(x)\psi_{q,k}(y),
    \qquad
    \Phi_k:=\sum_{p=0}^{N-q}\phi_{p,k}
    \in V_{N-q}^x.
\]
We may therefore apply Proposition~\ref{prop:1d-leakage}
to each $\Phi_k$ with $m=N-q$. Repeating the calculation in Step~1 yields
\[
\begin{aligned}
    \norm{(I-P_{\cS_N})A_xv}^2
    &=
    \sum_{k=1}^{d_q}
    \norm{(I_x-P_{V_{N-q}^x})D_N^x\Phi_k}_{L_x^2}^2
    \\
    &=
    (2^q-1)
    \sum_{k=1}^{d_q}
    \norm{P_{V_{N-q}^x}D_N^x\Phi_k}_{L_x^2}^2
    \\
    &=
    (2^q-1)\norm{P_{\cS_N}A_xv}^2.
\end{aligned}
\]
Thus, the same identity holds throughout
$V_{N-q}^x\otimes W_q^y$.

\underline{Step 3: Functions in the sparse grid space.} From the decomposition \eqref{eq:sparse-space}, we can write $\cS_N=\bigoplus_{q=0}^N V_{N-q}^x\otimes W_q^y$. Accordingly, for $u\in\cS_N$, we can decompose it as
\[
    u=\sum_{q=0}^N u^{(q)},
    \qquad
    u^{(q)}\in V_{N-q}^x\otimes W_q^y,
\]
Then, we have
\[
    P_{\cS_N}A_xu = \sum_{q=0}^N g_q, \qquad (I-P_{\cS_N})A_xu = \sum_{q=0}^N e_q,
\]
with
\[
    g_q:=P_{\cS_N}A_xu^{(q)},
    \qquad
    e_q:=(I-P_{\cS_N})A_xu^{(q)}.
\]
By Step~2,
\begin{equation*}
    \norm{e_q}^2 = (2^q-1)\norm{g_q}^2.
\end{equation*}

Since $u^{(q)}\in V_{N-q}^x\otimes W_q^y$, we have $A_xu^{(q)}\in V_N^x\otimes W_q^y$ and $P_{\cS_N}A_xu^{(q)}\in V_{N-q}^x\otimes W_q^y$. Therefore,
\[
    g_q\in V_{N-q}^x\otimes W_q^y, \qquad e_q\in V_N^x\otimes W_q^y.
\]
The orthogonality of the $y$-levels implies that $e_q\perp e_r$ whenever $q\neq r$. Then we have
\begin{equation}\label{eq:orthogonal-sum-2d-proof}
    \norm{(I-P_{\cS_N})A_xu}^2
    =
    \sum_{q=0}^N\norm{e_q}^2
    =
    \sum_{q=0}^N(2^q-1)\norm{g_q}^2.
\end{equation}

Next, we express $g_q$ in \eqref{eq:orthogonal-sum-2d-proof} in terms of the blocks $X_{pq}$. By the definition of $X_{pq}$, we compute
\[
    X_{pq}
    =
    P_{W_p^x\otimes W_q^y}A_xu
    =
    \sum_{r=0}^N
    P_{W_p^x\otimes W_q^y}A_xu^{(r)}
    =
    P_{W_p^x\otimes W_q^y}A_xu^{(q)}.
\]
In the last step, we have used the fact that $A_xu^{(r)}\in V_N^x\otimes W_r^y$ and the orthogonality of the $y$-levels implies that $P_{W_p^x\otimes W_q^y}A_xu^{(r)}=0$ for $r\neq q$.
In addition, by the definition of $g_q$, we have
\[
    g_q
    =
    P_{\cS_N}A_xu^{(q)}
    =
    \sum_{p+r\leq N}
    P_{W_p^x\otimes W_r^y}A_xu^{(q)}
    =
    \sum_{p=0}^{N-q}
    P_{W_p^x\otimes W_q^y}A_xu^{(q)}
    =
    \sum_{p=0}^{N-q}X_{pq}.
\]
By orthogonality,
\begin{equation}\label{eq:gq-norm-2d-proof}
    \norm{g_q}^2
    =
    \sum_{p=0}^{N-q}\norm{X_{pq}}^2.    
\end{equation}
Substituting \eqref{eq:gq-norm-2d-proof} into \eqref{eq:orthogonal-sum-2d-proof} gives
\[
    \norm{(I-P_{\cS_N})A_xu}^2
    =
    \sum_{q=0}^N(2^q-1)
    \sum_{p=0}^{N-q}\norm{X_{pq}}^2
    =
    \sum_{p+q\leq N}(2^q-1)\norm{X_{pq}}^2.
\]
This proves \eqref{eq:x-leakage}.
\end{proof}

\begin{proposition}[Zero blocks]
\label{prop:zero-axis-blocks}
For $u\in\cS_N$, let $X_{pq}$ and $Y_{pq}$ be defined as in Proposition~\ref{prop:2d-leakage}. Then
\begin{equation}\label{eq:zero-axis-blocks}
    X_{0q}=0, \qquad Y_{p0}=0, \qquad 0\leq p,q\leq N.
\end{equation}
\end{proposition}

\begin{proof}
A periodic $x$-difference has zero mean in the $x$-direction: for every fixed $j$,
\[
    \sum_i(A_xu)_{ij}
    =
    \sum_i(u_{ij}-u_{i-1,j})
    =
    0.
\]
Thus $A_xu$ has no $W_0^x$ component, so $X_{0q}=0$. Similarly,
$A_yu$ has no $W_0^y$ component, and hence $Y_{p0}=0$.
\end{proof}

\subsection{Sharp CFL condition}\label{subsec:sharp-cfl}

Now we are ready to prove the main result in this section.
\begin{theorem}[Sharp CFL condition for the sparse grid scheme in 2D]
\label{thm:2d-square}
    Let $N\geq1$, $c_1,c_2\geq0$, and $\max(c_1,c_2)>0$. The sparse grid scheme \eqref{eq:sparse-update} is stable at every time step,
    \[
        \norm{u^{n+1}}
        \leq
        \norm{u^n},
        \qquad
        n\geq0,
    \]
    if and only if
    \begin{equation}\label{eq:sharp-sparse-cfl}
        \Delta t
        \leq
        \frac{h}{\max(c_1,c_2)}.
    \end{equation}
\end{theorem}
\begin{proof}
Define $r_x := \frac{c_1\Delta t}{h}$ and $r_y := \frac{c_2\Delta t}{h}$.
We first prove necessity. Let $v = (v_{ij})\in\cF_N$ a function in the full grid space. Consider the $x$-alternating mode, which is defined by
\[
    v_{ij}=(-1)^i, \qquad 0\leq i,j<2^N.
\]
It actually belongs to $V_N^x\otimes W_0^y \subset\cS_N$ and satisfies
\[
    A_xv=2v,
    \qquad
    A_yv=0,
\]
and therefore after projection onto the sparse grid space,
\[
    P_{\cS_N}A_xv=2v,
    \qquad
    P_{\cS_N}A_yv=0.
\]
Taking $u^n=v$ in \eqref{eq:sparse-update} gives
\[
    u^{n+1}
    =
    (1-2r_x)v.
\]
Thus $L^2$ stability requires $\abs{1-2r_x}\leq1$, which is equivalent to $0\leq r_x\leq1$. Similarly, the $y$-alternating mode $v_{ij}=(-1)^j$ gives $0\leq r_y\leq1$. The conditions $r_x,r_y\leq1$ are equivalent to \eqref{eq:sharp-sparse-cfl}. Hence \eqref{eq:sharp-sparse-cfl} is necessary.

We next prove sufficiency. Assuming \eqref{eq:sharp-sparse-cfl} holds true, we have $0\leq r_x,r_y\leq1$. Since $u^n\in\cS_N$ and $P_{\cS_N}$ is the $L^2$ orthogonal projection onto $\cS_N$, Lemma~\ref{lem:D-energy} gives
\begin{equation}
    2\ip{u^n}{P_{\cS_N}A_xu^n}
    =
    2\ip{u^n}{A_xu^n}
    =
    \norm{A_xu^n}^2
    =
    \norm{P_{\cS_N}A_xu^n}^2
    +
    \norm{(I-P_{\cS_N})A_xu^n}^2.
    \label{eq:x-projected-energy}
\end{equation}
Similarly,
\begin{equation}
    2\ip{u^n}{P_{\cS_N}A_yu^n}
    =
    2\ip{u^n}{A_yu^n}
    =
    \norm{A_yu^n}^2
    =
    \norm{P_{\cS_N}A_yu^n}^2
    +
    \norm{(I-P_{\cS_N})A_yu^n}^2.
    \label{eq:y-projected-energy}
\end{equation}

Taking the $L^2$ norm of \eqref{eq:sparse-update}, we obtain
\begin{align}
    \norm{u^{n+1}}^2
    &=
    \norm{
        u^n
        -r_xP_{\cS_N}A_xu^n
        -r_yP_{\cS_N}A_yu^n
    }^2
    \notag\\
    &=
    \norm{u^n}^2
    +r_x^2\norm{P_{\cS_N}A_xu^n}^2
    +r_y^2\norm{P_{\cS_N}A_yu^n}^2
    \notag\\
    &\quad
    -2r_x\ip{u^n}{P_{\cS_N}A_xu^n}
    -2r_y\ip{u^n}{P_{\cS_N}A_yu^n}
    +2r_xr_y
    \ip{
        P_{\cS_N}A_xu^n
    }{
        P_{\cS_N}A_yu^n
    }.
    \label{eq:sparse-energy-expansion-first}
\end{align}
Substituting \eqref{eq:x-projected-energy} and \eqref{eq:y-projected-energy} into \eqref{eq:sparse-energy-expansion-first} gives
\begin{align}
    \norm{u^{n+1}}^2
    &=
    \norm{u^n}^2
    -r_x(1-r_x)
    \norm{P_{\cS_N}A_xu^n}^2
    -r_y(1-r_y)
    \norm{P_{\cS_N}A_yu^n}^2
    \notag\\
    &\quad
    -r_x
    \norm{(I-P_{\cS_N})A_xu^n}^2
    -r_y
    \norm{(I-P_{\cS_N})A_yu^n}^2
    +2r_xr_y
    \ip{
        P_{\cS_N}A_xu^n
    }{
        P_{\cS_N}A_yu^n
    }.
    \label{eq:sparse-energy-expansion}
\end{align}

Let $X_{pq}$ and $Y_{pq}$ be the projections of $A_xu^n$ and $A_yu^n$ onto $W_p^x\otimes W_q^y$, respectively, as defined in Proposition~\ref{prop:2d-leakage}. Then we have the estimate for the cross term in \eqref{eq:sparse-energy-expansion}:
\begin{align*}
    2
    \ip{
        P_{\cS_N}A_xu^n
    }{
        P_{\cS_N}A_yu^n
    }
    &=
    2
    \ip{\sum_{\substack{p+q\leq N}} X_{pq}}{\sum_{\substack{p'+q'\leq N}} Y_{p'q'}}
    \\    
    &=
    2
    \sum_{\substack{p+q\leq N}}
    \ip{X_{pq}}{Y_{pq}}
    \qquad \text{(orthogonality of $W_p^x\otimes W_q^y$)}
    \\    
    &=
    2
    \sum_{\substack{p+q\leq N\\p,q\geq1}}
    \ip{X_{pq}}{Y_{pq}}
    \qquad \text{(Proposition~\ref{prop:zero-axis-blocks})}
    \\
    &\leq
    \sum_{\substack{p+q\leq N\\p,q\geq1}}
    \left(
        \norm{X_{pq}}^2
        +
        \norm{Y_{pq}}^2
    \right)
    \\
    &\leq
    \sum_{p+q\leq N}
    \left(
        (2^q-1)\norm{X_{pq}}^2
        +
        (2^p-1)\norm{Y_{pq}}^2
    \right)
    \\
    &=
    \norm{(I-P_{\cS_N})A_xu^n}^2
    +
    \norm{(I-P_{\cS_N})A_yu^n}^2 \qquad \text{(Proposition~\ref{prop:2d-leakage})}
\end{align*}
Substituting this estimate into
\eqref{eq:sparse-energy-expansion} yields
\begin{align*}
    \norm{u^{n+1}}^2
    &\leq
    \norm{u^n}^2
    -r_x(1-r_x)
    \norm{P_{\cS_N}A_xu^n}^2
    -r_y(1-r_y)
    \norm{P_{\cS_N}A_yu^n}^2
    \\
    &\quad
    -r_x(1-r_y)
    \norm{(I-P_{\cS_N})A_xu^n}^2
    -r_y(1-r_x)
    \norm{(I-P_{\cS_N})A_yu^n}^2
    \\
    &\leq
    \norm{u^n}^2.
\end{align*}
This proves sufficiency.
\end{proof}

From the proof of Theorem~\ref{thm:2d-square}, we can derive an explicit formula for the spectral radius and operator norm of the sparse grid amplification matrix. We first introduce the following lemma which applies to any linear operator with eigenvalues at both endpoints of the unit disk.
\begin{lemma}\label{lem:endpoint-amplification}
Let $H$ be a linear operator on a finite dimensional Hilbert space such that $\norm{H}_2\leq1$ and both $1$ and $-1$ are eigenvalues of $H$. Then, for every $\mu\geq0$,
\begin{equation}
    \rho\bigl((1-\mu)I+\mu H\bigr)
    =
    \norm{(1-\mu)I+\mu H}_2
    =
    \max\{1,2\mu-1\}.
    \label{eq:endpoint-affine-scaling}
\end{equation}
\end{lemma}

\begin{proof}
For $0\leq \mu\leq1$, the triangle inequality gives $\norm{(1-\mu)I+\mu H}_2\leq1$. Since $1$ is an eigenvalue of $H$, $(1-\mu)I+\mu H$ has the eigenvalue $1$ as well. This leads to:
\begin{equation*}
    1\le \rho\bigl((1-\mu)I+\mu H\bigr)\le \norm{(1-\mu)I+\mu H}_2\le 1.
\end{equation*}
Hence, for $0\leq \mu\leq1$, we have
\[
    \rho\bigl((1-\mu)I+\mu H\bigr)=\norm{(1-\mu)I+\mu H}_2=1.
\]

For $\mu\geq1$, $\norm{(1-\mu)I+\mu H}_2\leq(\mu-1)+\mu\norm{H}_2\leq2\mu-1$. The eigenvalue $-1$ of $H$ produces the eigenvalue $(1-2\mu)$ of $(1-\mu)I+\mu H$. This leads to:
\begin{equation*}
    2\mu-1\le \rho\bigl((1-\mu)I+\mu H\bigr)\le \norm{(1-\mu)I+\mu H}_2\le 2\mu-1.
\end{equation*}
The proof is complete.
\end{proof}

Now we can derive the spectral radius and operator norm of the sparse grid amplification matrix.
\begin{corollary}\label{cor:2d-amplification-radius}
Under the assumptions of Theorem~\ref{thm:2d-square}, let $m:=\max(c_1,c_2)$ and $\nu:={\Delta t}/{h}$. Define the operators on $\cS_N$
\begin{align*}
    B:={}c_1P_{\cS_N}A_x+c_2P_{\cS_N}A_y,\qquad
    G(\nu):={} I-\nu B.
\end{align*}
Then, for any $\nu\geq0$,
\begin{align}
    \rho\bigl(G(\nu)\bigr)
    =\norm{G(\nu)}_2
    &=\max\{1,2m\nu-1\}.
    \label{eq:2d-sparse-amplification-radius}
\end{align}
\end{corollary}

\begin{proof}
Set $H:=G(1/m)$. Theorem \ref{thm:2d-square} gives $\norm{H}_2\leq1$. The constant mode gives the eigenvalue $1$ for $H$. 

Without of loss of generality, assume $c_1\geq c_2>0$. Consider the alternating mode in one direction with magnitude $m$, i.e., $u_{ij}= m(-1)^i$
\begin{equation*}
    H u = u - \frac{1}{m}Bu = u - \frac{1}{m}(c_1P_{\cS_N}A_xu + c_2P_{\cS_N}A_yu) = u - \frac{1}{m}(c_1 \cdot 2u + c_2 \cdot 0) = u - \frac{2c_1}{m}u = -u.
\end{equation*}
This gives the eigenvalue $-1$ for $H$.

Since $G(\nu)=(1-m\nu)I+m\nu H$, Lemma~\ref{lem:endpoint-amplification}, applied with $\mu=m\nu$, completes the proof.
\end{proof}

\begin{remark}[Generalization to high-order SSP RK methods]
    The above stability analysis can be easily extended from Euler forward method to high-order strong-stability-preserving (SSP) Runge-Kutta time discretizations \cite{gottlieb2001strong}. In such methods, the scheme can be written as a convex combination of the Euler forward, and we have the sufficient condition for the stability:
    \begin{equation*}
        \Delta t\le C_{\textrm{SSP}}\frac{h}{\max(c_1,c_2)},
    \end{equation*}
    where $C_{\textrm{SSP}}$ is the SSP coefficient \cite{gottlieb2001strong}. However, the sharpness of this condition is not guaranteed.
\end{remark}

\section{Stability analysis of the sparse grid scheme in higher dimensions}
\label{sec:proof-arbitrary-d}

In this section, we extend the stability analysis of the sparse grid scheme from 2D to higher dimensions with  $d\ge 2$. The proof uses the same one-dimensional leakage identity as in two dimensions. The new point is that the mixed terms between all pairs of coordinate directions must be estimated collectively.

\subsection{The sparse grid scheme in arbitrary dimensions}\label{subsec:d-setting}

Let $d\geq2$ and denote $\boldsymbol p=(p_1,\ldots,p_d)\in\mathbb N_0^d$ with $|\boldsymbol p|_1:=\sum_{\ell=1}^d p_\ell$. Define
\begin{equation*}
    W_{\boldsymbol p}:=\bigotimes_{\ell=1}^d W_{p_\ell}^{(\ell)},
\end{equation*}
where $W_{p_\ell}^{(\ell)}$ denotes the $p_\ell$-th level one-dimensional spaces in the $x_\ell$-direction. The full grid and sparse grid spaces are
\begin{align}
    \cF_N
    :=
    \bigotimes_{\ell=1}^d V_N^{(\ell)}
    =
    \bigoplus_{\boldsymbol p\in\{0,\ldots,N\}^d}W_{\boldsymbol p},\qquad
    \cS_N:=
    \bigoplus_{|\boldsymbol p|_1\leq N}W_{\boldsymbol p}.
    \label{eq:d-sparse-space}
\end{align}

For $1\leq\ell\leq d$, define the full grid difference operator in the $x_\ell$-direction $A_\ell: \cF_N \to \cF_N$ by
\begin{equation*}
    A_\ell := I_1\otimes \cdots \otimes I_{\ell-1}\otimes D_N^{(\ell)} \otimes I_{\ell+1}\otimes\cdots\otimes I_d.    
\end{equation*}
where $I_m$ is the identity operator in the $x_m$-direction. Consider the transport equation with constant coefficients in $d$ dimensions:
\begin{equation}
    u_t+\sum_{\ell=1}^d c_\ell u_{x_\ell}=0
    \qquad\text{on }\Omega = [0, 1]^d,
    \label{eq:d-transport}
\end{equation}
where $\boldsymbol c=(c_1,\ldots,c_d)$, $c_\ell\geq0$. With $h:=2^{-N}$ and $\Delta t\geq0$, the sparse grid forward Euler scheme is
\begin{equation}
    u^{n+1}
    =
    u^n
    -\sum_{\ell=1}^d
    \frac{c_\ell\Delta t}{h}
    P_{\cS_N}A_\ell u^n.
    \label{eq:d-sparse-update}
\end{equation}

\subsection{Multidimensional leakage and mixed terms}
\label{subsec:d-leakage}

The following proposition generalizes Proposition~\ref{prop:2d-leakage} in two dimensions to arbitrary dimensions. The proof is similar to the two-dimensional case, but uses a transverse decomposition of the sparse grid space in the direction of the difference operator, see \eqref{eq:d-transverse-decomposition} below.
\begin{proposition}[Multidimensional projection leakage identity]
\label{prop:d-leakage}
For any $u\in\cS_N$ and $1\leq\ell\leq d$, define \begin{equation}
    X_{\boldsymbol p}^{(\ell)}
    :=
    P_{W_{\boldsymbol p}}A_\ell u,
    \qquad
    \boldsymbol p\in \{0,\ldots,N\}^d,
    \qquad
    1\leq\ell\leq d.
    \label{eq:d-block-components}
\end{equation}
Then, we have
\begin{equation}
    \norm{(I-P_{\cS_N})A_\ell u}^2
    =
    \sum_{|\boldsymbol p|_1\leq N}
    \left(2^{|\boldsymbol p|_1-p_\ell}-1\right)
    \norm{X_{\boldsymbol p}^{(\ell)}}^2.
    \label{eq:d-leakage}
\end{equation}
Moreover,
\begin{equation}
    X_{\boldsymbol p}^{(\ell)}=0
    \qquad\text{whenever }p_\ell=0.
    \label{eq:d-zero-coordinate-block}
\end{equation}
\end{proposition}

\begin{proof}
We first introduce some notations. For a fixed direction $\ell$ and $\boldsymbol x = (x_j)_{j=1}^d\in\mathbb{R}^d$, write $\boldsymbol {x}_{-\ell}:=(x_j)_{j\neq\ell}\in\mathbb{R}^{d-1}$. 
For a given $\boldsymbol p\in\mathbb N_0^d$, denote $\boldsymbol p_{-\ell}\in\mathbb N_0^{d-1}$ the transverse multi-index obtained by removing the $\ell$-th component from $\boldsymbol p$, i.e., $\boldsymbol p_{-\ell}:=(p_m)_{m\neq\ell}$ and $|\boldsymbol p_{-\ell}|_1 = |\boldsymbol p|_1-p_\ell$.
For a transverse multi-index $\boldsymbol q=(q_j)_{j\neq\ell}\in\mathbb N_0^{d-1}$ with $|\boldsymbol q|_1\leq N$, define
\[
    W_{\boldsymbol q}^{(-\ell)}
    :=
    \bigotimes_{j\neq\ell}W_{q_j}^{(j)},
    \qquad
    |\boldsymbol q|_1:=\sum_{j\neq\ell}q_j,
    \qquad
    m_{\boldsymbol q}:=N-|\boldsymbol q|_1.
\]

We prove \eqref{eq:d-leakage} in three steps, following the argument of Proposition~\ref{prop:2d-leakage}.

\underline{Step 1: Functions in a single block.} We consider the subspace of $\cS_N$ corresponding to a fixed multi-index $\boldsymbol p\in\mathbb{N}_0^d$ with $|\boldsymbol p|_1\leq N$:
\[
    W_{\boldsymbol p}
    =
    W_{p_\ell}^{(\ell)}
    \otimes W_{\boldsymbol{p}_{-\ell}}^{(-\ell)}.
\]
We first consider a separable function
\[
    v(\boldsymbol{x}) = v(x_\ell,\boldsymbol{x}_{-\ell})
    =
    \phi(x_\ell)\psi(\boldsymbol{x}_{-\ell}),
    \qquad
    \phi\in W_{p_\ell}^{(\ell)},
    \quad
    \psi\in W_{\boldsymbol{p}_{-\ell}}^{(-\ell)}.
\]
Then $A_\ell$ acts only in the $x_\ell$-direction, i.e., $A_\ell v=(D_N^{(\ell)}\phi)\psi$. We compute
\[
\begin{aligned}
    P_{\cS_N} A_\ell v
    &=
    P_{\cS_N}((D_N^{(\ell)}\phi)\psi)
    \\
    &=
    \sum_{|\boldsymbol r|_1\leq N}
    P_{W_{r_\ell}^{(\ell)}\otimes W_{\boldsymbol r_{-\ell}}^{(-\ell)}}((D_N^{(\ell)}\phi)\psi)
    \\
    &=
    \sum_{|\boldsymbol r|_1\leq N}
    \left(P_{W_{r_\ell}^{(\ell)}}D_N^{(\ell)}\phi\right)
    \left(P_{W_{\boldsymbol r_{-\ell}}^{(-\ell)}}\psi\right)
    \\
    &=
    \sum_{r_{\ell}=0}^{m_{\boldsymbol q}}
    \left(P_{W_{r_{\ell}}^{(\ell)}}D_N^{(\ell)}\phi\right)\psi
    \\
    &=
    \left(P_{V_{m_{\boldsymbol q}}^{(\ell)}}D_N^{(\ell)}\phi\right)\psi.
\end{aligned}
\]
Here, in the fourth step, we used the fact that $\psi\in W_{\boldsymbol p_{-\ell}}^{(-\ell)}$ and $P_{W_{\boldsymbol r_{-\ell}}^{(-\ell)}}\psi=0$ for $\boldsymbol r_{-\ell}\neq \boldsymbol p_{-\ell}$. In the last step, $V_{m_{\boldsymbol q}}^{(\ell)}=\bigoplus_{k=0}^{m_{\boldsymbol q}}W_k^{(\ell)}$ is the one-dimensional space at level $m_{\boldsymbol q}$ in the variable $x_\ell$. 
Similarly,
\[
    (I-P_{\cS_N})A_\ell v
    =
    \left(
        \left(I_\ell-P_{V_{m_{\boldsymbol q}}^{(\ell)}}\right)
        D_N^{(\ell)}\phi
    \right)\psi.
\]
Because
$\phi\in W_{p_\ell}^{(\ell)}
\subseteq V_{m_{\boldsymbol q}}^{(\ell)}$,
Proposition~\ref{prop:1d-leakage}, applied with
$m=m_{\boldsymbol q}$, gives
\[
    \norm{
        \left(I_\ell-P_{V_{m_{\boldsymbol q}}^{(\ell)}}\right)
        D_N^{(\ell)}\phi
    }_{L^2_{x_\ell}}^2
    =
    \left(2^{|\boldsymbol q|_1}-1\right)
    \norm{
        P_{V_{m_{\boldsymbol q}}^{(\ell)}}
        D_N^{(\ell)}\phi
    }_{L^2_{x_\ell}}^2.
\]
Multiplying both sides by $\norm{\psi}_{L^2_{x_{-\ell}}}^2$ gives
\[
    \norm{(I-P_{\cS_N})A_\ell v}^2
    =
    \left(2^{|\boldsymbol q|_1}-1\right)
    \norm{P_{\cS_N}A_\ell v}^2.
\]

Next, let $v\in W_{\boldsymbol p}$. Choose an orthonormal basis $\{\psi_{\boldsymbol q,\alpha}\}_{\alpha=1}^{d_{\boldsymbol q}}$ of $W_{\boldsymbol q}^{(-\ell)}$ with $d_{\boldsymbol q}:=\dim W_{\boldsymbol q}^{(-\ell)}$, and expand $v$ as
\[
    v(\boldsymbol{x}) = v(x_\ell,\boldsymbol{x}_{-\ell})=\sum_{\alpha=1}^{d_{\boldsymbol q}}
    \phi_\alpha(x_\ell)\psi_{\boldsymbol q,\alpha}(\boldsymbol{x}_{-\ell}),
    \qquad
    \phi_\alpha\in W_{p_\ell}^{(\ell)}.
\]
Applying the projection to each term gives
\[
    P_{\cS_N}A_\ell v
    =
    \sum_{\alpha=1}^{d_{\boldsymbol q}}
    \left(
        P_{V_{m_{\boldsymbol q}}^{(\ell)}}
        D_N^{(\ell)}\phi_\alpha
    \right)\psi_{\boldsymbol q,\alpha},
    \qquad
    (I-P_{\cS_N})A_\ell v
    =
    \sum_{\alpha=1}^{d_{\boldsymbol q}}
    \left(
        \left(I_\ell-P_{V_{m_{\boldsymbol q}}^{(\ell)}}\right)
        D_N^{(\ell)}\phi_\alpha
    \right)\psi_{\boldsymbol q,\alpha}.
\]
Orthonormality eliminates the cross terms when taking squared norms. Consequently,
\[
\begin{aligned}
    \norm{(I-P_{\cS_N})A_\ell v}^2
    &=
    \sum_{\alpha=1}^{d_{\boldsymbol q}}
    \norm{
        \left(I_\ell-P_{V_{m_{\boldsymbol q}}^{(\ell)}}\right)
        D_N^{(\ell)}\phi_\alpha
    }_{L^2_{x_\ell}}^2
    \\
    &=
    \left(2^{|\boldsymbol q|_1}-1\right)
    \sum_{\alpha=1}^{d_{\boldsymbol q}}
    \norm{
        P_{V_{m_{\boldsymbol q}}^{(\ell)}}
        D_N^{(\ell)}\phi_\alpha
    }_{L^2_{x_\ell}}^2
    \\
    &=
    \left(2^{|\boldsymbol q|_1}-1\right)
    \norm{P_{\cS_N}A_\ell v}^2.
\end{aligned}
\]
This identity holds on any single block $W_{\boldsymbol p}$ and its factor depends only on the transverse index $\boldsymbol q = \boldsymbol p_{-\ell}$.

\underline{Step 2: Functions with a fixed transverse index.}
We next extend the identity to
\[
    v\in
    V_{m_{\boldsymbol q}}^{(\ell)}
    \otimes W_{\boldsymbol q}^{(-\ell)}
    =
    \bigoplus_{k=0}^{m_{\boldsymbol q}}
    W_k^{(\ell)}\otimes W_{\boldsymbol q}^{(-\ell)}.
\]
Write
\[
    v=\sum_{k=0}^{m_{\boldsymbol q}}v_k,
    \qquad
    v_k\in W_k^{(\ell)}\otimes W_{\boldsymbol q}^{(-\ell)}.
\]
Using the same orthonormal basis $\{\psi_{\boldsymbol q,\alpha}\}_{\alpha=1}^{d_{\boldsymbol q}}$, expand
\[
    v_k(x_\ell,x_{-\ell})
    =
    \sum_{\alpha=1}^{d_{\boldsymbol q}}
    \phi_{k,\alpha}(x_\ell)
    \psi_{\boldsymbol q,\alpha}(x_{-\ell}),
    \qquad
    \phi_{k,\alpha}\in W_k^{(\ell)}.
\]
Collecting the coefficients of each
$\psi_{\boldsymbol q,\alpha}$ gives
\[
    v(x_\ell,x_{-\ell})
    =
    \sum_{\alpha=1}^{d_{\boldsymbol q}}
    \Phi_\alpha(x_\ell)\psi_{\boldsymbol q,\alpha}(x_{-\ell}),
    \qquad
    \Phi_\alpha
    :=
    \sum_{k=0}^{m_{\boldsymbol q}}\phi_{k,\alpha}
    \in V_{m_{\boldsymbol q}}^{(\ell)}.
\]
Thus,
\[
\begin{aligned}
    \norm{(I-P_{\cS_N})A_\ell v}^2
    &=
    \sum_{\alpha=1}^{d_{\boldsymbol q}}
    \norm{
        \left(I_\ell-P_{V_{m_{\boldsymbol q}}^{(\ell)}}\right)
        D_N^{(\ell)}\Phi_\alpha
    }_{L^2_{x_\ell}}^2
    \\
    &=
    \left(2^{|\boldsymbol q|_1}-1\right)
    \sum_{\alpha=1}^{d_{\boldsymbol q}}
    \norm{
        P_{V_{m_{\boldsymbol q}}^{(\ell)}}
        D_N^{(\ell)}\Phi_\alpha
    }_{L^2_{x_\ell}}^2
    \\
    &=
    \left(2^{|\boldsymbol q|_1}-1\right)
    \norm{P_{\cS_N}A_\ell v}^2.
\end{aligned}
\]
The same identity therefore holds for $V_{m_{\boldsymbol q}}^{(\ell)} \otimes W_{\boldsymbol q}^{(-\ell)}$.

\underline{Step 3: Functions in the sparse grid space.}
When $\boldsymbol p_{-\ell}=\boldsymbol q$ is fixed, the constraint $|\boldsymbol p|_1\leq N$ is equivalent to $0\leq p_\ell\leq m_{\boldsymbol q}$. Decompose the sparse grid space $\cS_N$ into
\begin{equation*}
\begin{aligned}
    \cS_N
    =
    \bigoplus_{|\boldsymbol q|_1\leq N}
    \left(
        \bigoplus_{k=0}^{m_{\boldsymbol q}}W_k^{(\ell)}
    \right)\otimes W_{\boldsymbol q}^{(-\ell)}
    =
    \bigoplus_{|\boldsymbol q|_1\leq N}
    V_{m_{\boldsymbol q}}^{(\ell)}
    \otimes W_{\boldsymbol q}^{(-\ell)}.
\end{aligned}
\label{eq:d-transverse-decomposition}
\end{equation*}
Accordingly, for $u\in\cS_N$, write
\[
    u=\sum_{|\boldsymbol q|_1\leq N}u^{(\boldsymbol q)},
    \qquad
    u^{(\boldsymbol q)}
    \in
    V_{m_{\boldsymbol q}}^{(\ell)}
    \otimes W_{\boldsymbol q}^{(-\ell)}.
\]
Define
\[
    g_{\boldsymbol q}
    :=P_{\cS_N}A_\ell u^{(\boldsymbol q)},
    \qquad
    e_{\boldsymbol q}
    :=(I-P_{\cS_N})A_\ell u^{(\boldsymbol q)}.
\]
By linearity,
\[
    P_{\cS_N}A_\ell u
    =
    \sum_{|\boldsymbol q|_1\leq N}g_{\boldsymbol q},
    \qquad
    (I-P_{\cS_N})A_\ell u
    =
    \sum_{|\boldsymbol q|_1\leq N}e_{\boldsymbol q}.
\]
Step~2 gives
\begin{equation*}
    \norm{e_{\boldsymbol q}}^2
    =
    \left(2^{|\boldsymbol q|_1}-1\right)
    \norm{g_{\boldsymbol q}}^2.
\end{equation*}

Since $A_\ell$ leaves the transverse factors unchanged,
and the projection formulas in Step~1 preserve their indices,
\[
    g_{\boldsymbol q}
    \in
    V_{m_{\boldsymbol q}}^{(\ell)}
    \otimes W_{\boldsymbol q}^{(-\ell)},
    \qquad
    e_{\boldsymbol q}
    \in
    V_N^{(\ell)}\otimes W_{\boldsymbol q}^{(-\ell)}.
\]
Distinct transverse spaces are orthogonal, so
$e_{\boldsymbol q}\perp e_{\boldsymbol r}$ whenever
$\boldsymbol q\neq\boldsymbol r$.
Therefore,
\begin{equation}
    \norm{(I-P_{\cS_N})A_\ell u}^2
    =
    \sum_{|\boldsymbol q|_1\leq N}\norm{e_{\boldsymbol q}}^2
    =
    \sum_{|\boldsymbol q|_1\leq N}
    \left(2^{|\boldsymbol q|_1}-1\right)
    \norm{g_{\boldsymbol q}}^2.
    \label{eq:orthogonal-sum-d-proof}
\end{equation}

It remains to express $g_{\boldsymbol q}$ in terms of
$X_{\boldsymbol p}^{(\ell)}$.
For $\boldsymbol p_{-\ell}=\boldsymbol q$, we have
\[
    X_{\boldsymbol p}^{(\ell)}
    =
    P_{W_{\boldsymbol p}}A_\ell u
    =
    \sum_{|\boldsymbol r|_1\leq N}
    P_{W_{\boldsymbol p}}A_\ell u^{(\boldsymbol r)}
    =
    P_{W_{\boldsymbol p}}A_\ell u^{(\boldsymbol q)}.
\]
Indeed,
$A_\ell u^{(\boldsymbol r)}
\in V_N^{(\ell)}\otimes W_{\boldsymbol r}^{(-\ell)}$,
which is orthogonal to $W_{\boldsymbol p}$ when
$\boldsymbol r\neq\boldsymbol q$.
Using the orthogonal decomposition of $\cS_N$ once more,
\[
    g_{\boldsymbol q}
    =
    P_{\cS_N}A_\ell u^{(\boldsymbol q)}
    =
    \sum_{|\boldsymbol p|_1\leq N}
    P_{W_{\boldsymbol p}}A_\ell u^{(\boldsymbol q)}
    =
    \sum_{\substack{
        |\boldsymbol p|_1\leq N\\
        \boldsymbol p_{-\ell}=\boldsymbol q}}
    X_{\boldsymbol p}^{(\ell)}.
\]
The Haar blocks in this sum are mutually orthogonal.
Hence,
\[
    \norm{g_{\boldsymbol q}}^2
    =
    \sum_{\substack{
        |\boldsymbol p|_1\leq N\\
        \boldsymbol p_{-\ell}=\boldsymbol q}}
    \norm{X_{\boldsymbol p}^{(\ell)}}^2.
\]
Substituting this identity into
\eqref{eq:orthogonal-sum-d-proof} gives
\[
    \norm{(I-P_{\cS_N})A_\ell u}^2
    =
    \sum_{|\boldsymbol q|_1\leq N}
    \left(2^{|\boldsymbol q|_1}-1\right)
    \sum_{\substack{
        |\boldsymbol p|_1\leq N\\
        \boldsymbol p_{-\ell}=\boldsymbol q}}
    \norm{X_{\boldsymbol p}^{(\ell)}}^2
    =
    \sum_{|\boldsymbol p|_1\leq N}
    \left(2^{|\boldsymbol p|_1-p_\ell}-1\right)
    \norm{X_{\boldsymbol p}^{(\ell)}}^2,
\]
where the last equality uses
$|\boldsymbol q|_1=|\boldsymbol p|_1-p_\ell$.
This proves \eqref{eq:d-leakage}.

Finally, with $h=2^{-N}$ and periodic extension,
\[
    (A_\ell u)(x_\ell,x_{-\ell})
    =
    u(x_\ell,x_{-\ell})
    -
    u(x_\ell-h,x_{-\ell}).
\]
Periodicity implies
\[
    \int_0^1
    (A_\ell u)(x_\ell,x_{-\ell})\,dx_\ell=0
\]
for almost every $x_{-\ell}$.
If $p_\ell=0$, every $w\in W_{\boldsymbol p}$
is independent of $x_\ell$, since
$W_0^{(\ell)}=\operatorname{span}\{1\}$.
Integrating first in $x_\ell$ therefore gives
$\ip{A_\ell u}{w}=0$.
Hence
$X_{\boldsymbol p}^{(\ell)}
=P_{W_{\boldsymbol p}}A_\ell u=0$,
which proves \eqref{eq:d-zero-coordinate-block}.
\end{proof}

\begin{proposition}[Mixed term estimate in arbitrary dimensions]
\label{prop:d-mixed-terms}
Let $u\in\cS_N$ and $0\leq r_\ell\leq1$ for
$1\leq\ell\leq d$. Then
\begin{equation}
    2\sum_{1\leq\ell<m\leq d}r_\ell r_m
    \ip{P_{\cS_N}A_\ell u}
       {P_{\cS_N}A_m u}
    \leq
    \sum_{\ell=1}^d r_\ell
    \norm{(I-P_{\cS_N})A_\ell u}^2.
    \label{eq:d-mixed-leakage}
\end{equation}
\end{proposition}

\begin{proof}
For $|\boldsymbol p|_1\leq N$, define
\[
    J(\boldsymbol p) := \{\ell:p_\ell\geq1\}.
\]

First, with $X_{\boldsymbol p}^{(\ell)}$ defined in \eqref{eq:d-block-components}, we separate the contributions of different blocks. By orthogonality,
\[
    \ip{P_{\cS_N}A_\ell u}{P_{\cS_N}A_m u}
    =
    \sum_{|\boldsymbol p|_1\leq N}
    \ip{X_{\boldsymbol p}^{(\ell)}}
       {X_{\boldsymbol p}^{(m)}}.
\]
Moreover, \eqref{eq:d-zero-coordinate-block} implies that only directions in $J(\boldsymbol p)$ contribute on $W_{\boldsymbol p}$.
Thus,
\begin{equation*}
\begin{aligned}
    2\sum_{1\leq\ell<m\leq d}
    r_\ell r_m
    \ip{P_{\cS_N}A_\ell u}{P_{\cS_N}A_m u}
    &=
    \sum_{|\boldsymbol p|_1\leq N}
    2\sum_{\substack{\ell<m\\
                     \ell,m\in J(\boldsymbol p)}}
    r_\ell r_m
    \ip{X_{\boldsymbol p}^{(\ell)}}
       {X_{\boldsymbol p}^{(m)}} \\
    &\le
    \sum_{|\boldsymbol p|_1\leq N}
    \sum_{\substack{\ell<m\\
                     \ell,m\in J(\boldsymbol p)}}
    r_\ell r_m
    \left(
        \norm{X_{\boldsymbol p}^{(\ell)}}^2
        +
        \norm{X_{\boldsymbol p}^{(m)}}^2
    \right) \\
    &=
    \sum_{|\boldsymbol p|_1\leq N}
    \sum_{\ell\in J(\boldsymbol p)}
    r_\ell
    \left(
        \sum_{\substack{m\in J(\boldsymbol p)\\m\neq\ell}}
        r_m
    \right)
    \norm{X_{\boldsymbol p}^{(\ell)}}^2.
\end{aligned}    
\end{equation*}
Define $s_{\boldsymbol p}:=\#J(\boldsymbol p)$. There are $s_{\boldsymbol p}-1$ terms in the inner sum,
and each satisfies $0\leq r_m\leq1$.
Therefore,
\begin{equation*}
    2\sum_{1\leq\ell<m\leq d}r_\ell r_m
    \ip{P_{\cS_N}A_\ell u}
       {P_{\cS_N}A_m u} 
       \le 
       \sum_{|\boldsymbol p|_1\leq N}
    \sum_{\ell\in J(\boldsymbol p)}
    r_\ell(s_{\boldsymbol p}-1)
    \norm{X_{\boldsymbol p}^{(\ell)}}^2.
\end{equation*}

By the definition of $J(\boldsymbol p)$, we have $p_\ell\geq1$ for any $\ell\in J(\boldsymbol p)$. Hence, for any $\ell\in J(\boldsymbol p)$,
\[
    s_{\boldsymbol p}-1
    \leq \sum_{m\neq\ell}p_m
    =|\boldsymbol p|_1-p_\ell\leq
    2^{|\boldsymbol p|_1-p_\ell}-1.
\]
The last inequality follows from $k\leq2^k-1$ for any integer $k\geq0$. Using \eqref{eq:d-leakage}, we conclude that
\begin{equation*}
    2\sum_{1\leq\ell<m\leq d}
    r_\ell r_m
    \ip{P_{\cS_N}A_\ell u}{P_{\cS_N}A_m u}    
    \leq
    \sum_{\ell=1}^d r_\ell
    \sum_{|\boldsymbol p|_1\leq N}
    \left(2^{|\boldsymbol p|_1-p_\ell}-1\right)
    \norm{X_{\boldsymbol p}^{(\ell)}}^2
    =
    \sum_{\ell=1}^d
    r_\ell\norm{(I-P_{\cS_N})A_\ell u}^2.
\end{equation*}
\end{proof}

\subsection{Sharp CFL condition in arbitrary dimensions}
\label{subsec:d-sharp-cfl}

\begin{theorem}[Sharp CFL condition in arbitrary dimensions]
\label{thm:d-hypercube}
Let $d\geq2$, $N\geq1$, $c_\ell\geq0$, and $\max_{1\leq\ell\leq d}c_\ell>0$. The sparse grid scheme \eqref{eq:d-sparse-update} is stable for every $u^n\in\cS_N$,
\[
    \norm{u^{n+1}}\leq\norm{u^n},
    \qquad
    n\geq0,
\]
if and only if
\begin{equation}
    \Delta t
    \leq
    \frac{h}{\displaystyle\max_{1\leq\ell\leq d}c_\ell}.
    \label{eq:d-sharp-sparse-cfl}
\end{equation}
\end{theorem}

\begin{proof}
Set 
\begin{equation*}
    r_\ell:=\frac{c_\ell\Delta t}{h},
\end{equation*}
for $1\leq\ell\leq d$. We first prove necessity. Fix a direction $\ell$ and let $v_{\boldsymbol i}:=(-1)^{i_\ell}$ be the alternating mode on the full grid indexed by $\boldsymbol i=(i_1,\ldots,i_d)$. Then
\[
    v\in
    W_N^{(\ell)}\otimes
    \bigotimes_{m\neq\ell}W_0^{(m)}
    \subset\cS_N,
\]
and
\[
    P_{\cS_N}A_\ell v=2v,
    \qquad
    P_{\cS_N}A_m v=0
    \quad(m\neq\ell).
\]
Taking $u^n=v$ in \eqref{eq:d-sparse-update} gives $u^{n+1}=(1-2r_\ell)v$. Hence the stability condition requires $|1-2r_\ell|\leq1$, or equivalently, $0\leq r_\ell\leq1$. Applying this argument in every direction shows that
\eqref{eq:d-sharp-sparse-cfl} is necessary.

We next prove sufficiency. Assume \eqref{eq:d-sharp-sparse-cfl} holds true, so $0\leq r_\ell\leq1$ for every $1\le\ell\le d$. Since $u^n\in\cS_N$, the one-dimensional periodic difference identity in Lemma~\ref{lem:D-energy}
,
applied in the $x_\ell$-direction, gives
\begin{align}
    2\ip{u^n}{P_{\cS_N}A_\ell u^n}
    &=
    2\ip{u^n}{A_\ell u^n}
    =
    \norm{A_\ell u^n}^2
    =
    \norm{P_{\cS_N}A_\ell u^n}^2
    +
    \norm{(I-P_{\cS_N})A_\ell u^n}^2.
    \label{eq:d-projected-energy}
\end{align}
Taking the norm on both sides of the scheme \eqref{eq:d-sparse-update}, we obtain
\begin{align*}
    \norm{u^{n+1}}^2
    &=
    \norm{
        u^n-
        \sum_{\ell=1}^d
        r_\ell P_{\cS_N}A_\ell u^n
    }^2
    \\
    &=
    \norm{u^n}^2 + \sum_{\ell=1}^{d}
    r_\ell^2 \norm{P_{\cS_N}A_\ell u^n}^2
    -2\sum_{\ell=1}^d r_\ell \ip{u^n}{P_{\cS_N}A_\ell u^n} +2\sum_{1\leq\ell<m\leq d}
    r_\ell r_m
    \ip{P_{\cS_N}A_\ell u^n}{P_{\cS_N}A_m u^n}
    \\
    &=
    \norm{u^n}^2
    -\sum_{\ell=1}^d
    r_\ell(1-r_\ell)
    \norm{P_{\cS_N}A_\ell u^n}^2    
    -\sum_{\ell=1}^d
    r_\ell
    \norm{(I-P_{\cS_N})A_\ell u^n}^2
    +2\sum_{1\leq\ell<m\leq d}
    r_\ell r_m
    \ip{P_{\cS_N}A_\ell u^n}
       {P_{\cS_N}A_m u^n}
    \\
    &\leq
    \norm{u^n}^2
    -\sum_{\ell=1}^d
    r_\ell(1-r_\ell)
    \norm{P_{\cS_N}A_\ell u^n}^2 \\
    &\leq\norm{u^n}^2.
\end{align*}
Here we use \eqref{eq:d-projected-energy} and Lemma~\ref{prop:d-mixed-terms}.
This proves sufficiency.
\end{proof}

The following result extends the sparse grid result in Corollary~\ref{cor:2d-amplification-radius} to arbitrary dimensions. The proof is similar to the two-dimensional case, so we omit the details.
\begin{corollary}
\label{cor:d-amplification-radius}
Under the assumptions of Theorem~\ref{thm:d-hypercube}, set
\[
    m:=\max_{1\leq\ell\leq d}c_\ell,
    \qquad
    B_{\boldsymbol c}
    :=
    \sum_{\ell=1}^d c_\ell P_{\cS_N}A_\ell,
    \qquad
    G_{\boldsymbol c}(\nu)
    :=
    I-\nu B_{\boldsymbol c},
\]
where \(\nu=\Delta t/h\). Then, for every \(\nu\geq0\),
\begin{equation}
    \rho\bigl(G_{\boldsymbol c}(\nu)\bigr)
    =
    \norm{G_{\boldsymbol c}(\nu)}_2
    =
    \max\{1,2m\nu-1\}.
    \label{eq:d-exact-amplification-radius}
\end{equation}
\end{corollary}

\section{CFL conditions on general downward closed index sets}
\label{sec:downward-closed}

The sparse grid space $\cS_N$ defined in
\eqref{eq:d-sparse-space} belongs to a broader family of hierarchical spaces. For example, anisotropic sparse grids use weighted level sets \cite{bungartz2004sparse,griebel2000optimized}, while adaptive sparse grid DG methods adjust the index set during time evolution \cite{guo2017adaptive}. These constructions motivate an extension of the stability analysis in Section~\ref{sec:proof-arbitrary-d} beyond the standard total level constraint in \eqref{eq:d-sparse-space}.

In this section, we consider a general downward closed index set and relate its geometry to the admissible time step. The key observation is that, along each coordinate direction, downward closedness produces complete one-dimensional slices. The compression and leakage identities from the preceding sections therefore remain applicable, with weights determined by the maximum level on each slice. We use these identities to derive an explicit sufficient CFL condition in Theorem~\ref{thm:dc-sufficient}, and then identify a geometric criterion for its sharpness in Theorem~\ref{thm:dc-sharpness}. Full tensor product grids and the standard sparse grid satisfy this criterion, recovering their respective sharp CFL conditions from a single formula. An L-shaped example illustrates that the sufficient condition need not be sharp for a general downward closed set.

\subsection{Hierarchical spaces and the scheme}
\label{subsec:dc-setting}

\begin{definition}[Downward closed set]
    A finite index set $\Lambda\subset\mathbb N_0^d$ is called \emph{downward closed} if $\boldsymbol p\in\Lambda$ and $\boldsymbol0\leq\boldsymbol q\leq\boldsymbol p$ componentwise imply $\boldsymbol q\in\Lambda$.
\end{definition}

Throughout this section, let $\Lambda\subset\mathbb N_0^d$ be a finite, nonempty, and downward closed set. The associated hierarchical space is
\[
    \cS_\Lambda
    :=\bigoplus_{\boldsymbol p\in\Lambda}W_{\boldsymbol p}.
\]
Its maximum directional levels and finest mesh sizes are
\[
    N_\ell:=\max_{\boldsymbol p\in\Lambda}p_\ell,
    \qquad
    h_\ell:=2^{-N_\ell},
    \qquad 1\leq\ell\leq d.
\]
With $\boldsymbol N=(N_1,\ldots,N_d)$, define the full grid space and the directional difference operators by
\[
    \cF_{\boldsymbol N}
    :=\bigotimes_{m=1}^d V_{N_m}^{(m)},
    \qquad
    A_\ell
    :=I_1\otimes\cdots\otimes I_{\ell-1}
      \otimes D_{N_\ell}^{(\ell)}
      \otimes I_{\ell+1}\otimes\cdots\otimes I_d,
\]
where $I_m$ is the identity on $V_{N_m}^{(m)}$. Then $\cS_\Lambda\subseteq\cF_{\boldsymbol N}$, and the piecewise constant upwind DG scheme with forward Euler time stepping for \eqref{eq:d-transport} takes the form
\begin{equation}
    u^{n+1}
    =
    u^n-\sum_{\ell=1}^d
    r_\ell P_{\cS_\Lambda}A_\ell u^n,
    \qquad
    r_\ell:=\frac{c_\ell\Delta t}{h_\ell}.
    \label{eq:dc-scheme}
\end{equation}

To describe how the projection acts in a fixed direction, we group hierarchical blocks by their transverse indices. For $\boldsymbol p\in\mathbb N_0^d$, write $\boldsymbol p_{-\ell}:=(p_m)_{m\neq\ell}$. Conversely, for $k\in\mathbb N_0$ and $\boldsymbol q=(q_m)_{m\neq\ell}\in\mathbb N_0^{d-1}$, let $(k;\boldsymbol q)_\ell$ denote the multi-index with $\ell$th component $k$ and remaining components $q_m$. We also write
\[
    W_{\boldsymbol q}^{(-\ell)}
    :=\bigotimes_{m\neq\ell}W_{q_m}^{(m)}.
\]

The transverse indices occurring in $\Lambda$ form the set
\[
    \mathcal Q_\ell
    :=
    \bigl\{
        \boldsymbol q\in\mathbb N_0^{d-1}:
        (0;\boldsymbol q)_\ell\in\Lambda
    \bigr\}.
\]
For $\boldsymbol q\in\mathcal Q_\ell$, define the maximum
level on the corresponding slice by
\[
    K_\ell(\boldsymbol q)
    :=
    \max\bigl\{
        k\in\mathbb N_0:
        (k;\boldsymbol q)_\ell\in\Lambda
    \bigr\}.
\]
Downward closedness ensures that the levels on this slice are exactly $0,\ldots,K_\ell(\boldsymbol q)$. Thus, the corresponding space is $V_{K_\ell(\boldsymbol q)}^{(\ell)} \otimes W_{\boldsymbol q}^{(-\ell)}$. This complete one dimensional hierarchy will allow us to apply the leakage identity separately on each slice, as in the standard sparse grid case.

As an example, on the standard sparse grid space \eqref{eq:d-sparse-space}, the transverse index set and the maximum slice level are
\[
    \mathcal Q_\ell
    =
    \bigl\{
        \boldsymbol q\in\mathbb N_0^{d-1}:
        |\boldsymbol q|_1\leq N
    \bigr\},
    \qquad
    K_\ell(\boldsymbol q)=N-|\boldsymbol q|_1.
\]
For the full grid index set $\Lambda=\prod_{\ell=1}^d\{0,\ldots,N_\ell\}$, every admissible slice instead has $K_\ell(\boldsymbol q)=N_\ell$.
For a general downward closed set, the slice levels may vary with $\boldsymbol q$, but they always satisfy
$K_\ell(\boldsymbol0)=N_\ell$ and the following
monotonicity property.

\begin{lemma}[Monotonicity of slice levels]
\label{lem:dc-monotone}
Let $\boldsymbol q\leq\boldsymbol q'$ componentwise with
$\boldsymbol q'\in\mathcal Q_\ell$. Then
$\boldsymbol q\in\mathcal Q_\ell$ and
$K_\ell(\boldsymbol q)\geq K_\ell(\boldsymbol q')$.
\end{lemma}

\begin{proof}
Set $k:=K_\ell(\boldsymbol q')$. Since
$(k;\boldsymbol q')_\ell\in\Lambda$ and
$(k;\boldsymbol q)_\ell\leq(k;\boldsymbol q')_\ell$,
downward closedness gives $(k;\boldsymbol q)_\ell\in\Lambda$.
It follows that $(0;\boldsymbol q)_\ell\in\Lambda$ and
$K_\ell(\boldsymbol q)\geq k$, proving both claims.
\end{proof}

\subsection{Projection leakage identities on general downward closed sets}
\label{subsec:dc-leakage}

The preceding slice structure determines exactly how much energy a directional difference loses under projection onto $\cS_\Lambda$. In the standard sparse grid, this loss is controlled by the transverse total level. For a general downward closed set, the relevant quantity is the difference between $N_\ell$ and the slice level $K_\ell(\boldsymbol p_{-\ell})$. The following proposition extends Proposition~\ref{prop:d-leakage}.

\begin{proposition}[Projection leakage formulas on downward closed sets]
\label{prop:dc-leakage}
Let $\Lambda\subset\mathbb N_0^d$ be a finite, nonempty, and downward closed set. For $u\in\cS_\Lambda$ and $1\leq\ell\leq d$, define
\[
    X_{\boldsymbol p}^{(\ell)}
    :=
    P_{W_{\boldsymbol p}}A_\ell u,
    \qquad
    \boldsymbol p\in\Lambda.
\]
Then
\begin{equation}
    \norm{(I-P_{\cS_\Lambda})A_\ell u}^2
    =
    \sum_{\boldsymbol p\in\Lambda}
    \left(2^{N_\ell-K_\ell(\boldsymbol p_{-\ell})}-1\right)
    \norm{X_{\boldsymbol p}^{(\ell)}}^2.
    \label{eq:dc-leakage}
\end{equation}
Moreover, $X_{\boldsymbol p}^{(\ell)}=0$
whenever $p_\ell=0$.
\end{proposition}

\begin{proof}
Fix a direction $\ell$ and write $\boldsymbol{x}_{-\ell}:=(x_j)_{j\neq\ell}$. For $\boldsymbol q\in\mathcal Q_\ell$, set $m_{\boldsymbol q}:=K_\ell(\boldsymbol q)$. By downward closedness,
\[
    \bigl\{
        k\in\mathbb N_0:(k;\boldsymbol q)_\ell\in\Lambda
    \bigr\}
    =
    \{0,1,\ldots,m_{\boldsymbol q}\},
    \qquad
    0\leq m_{\boldsymbol q}\leq N_\ell.
\]
Following Proposition~\ref{prop:d-leakage}, we first start with a single Haar block, then combine the blocks on a fixed transverse slice, and finally sum over the orthogonal slices.

\underline{Step 1: Functions in a single Haar block.} Fix $\boldsymbol p\in\Lambda$ and set $\boldsymbol q:=\boldsymbol p_{-\ell}$. Then $\boldsymbol q\in\mathcal Q_\ell$, $p_\ell\leq m_{\boldsymbol q}$, and
\[
    W_{\boldsymbol p}
    =
    W_{p_\ell}^{(\ell)}
    \otimes W_{\boldsymbol q}^{(-\ell)}.
\]
We first consider a separable function
\[
    v(\boldsymbol{x}) = v(x_\ell,\boldsymbol{x}_{-\ell})
    =
    \phi(x_\ell)\psi(\boldsymbol{x}_{-\ell}),
    \qquad
    \phi\in W_{p_\ell}^{(\ell)},
    \quad
    \psi\in W_{\boldsymbol q}^{(-\ell)}.
\]
Since $A_\ell$ acts only in the $x_\ell$-direction, $A_\ell v=(D_{N_\ell}^{(\ell)}\phi)\psi$. For any $f\in V_{N_\ell}^{(\ell)}$, the orthogonal decomposition of $\cS_\Lambda$ gives
\[
\begin{aligned}
    P_{\cS_\Lambda}(f\psi)
    =
    \sum_{\boldsymbol r\in\Lambda}
    \left(P_{W_{r_\ell}^{(\ell)}}f\right)
    \left(
        P_{W_{\boldsymbol r_{-\ell}}^{(-\ell)}}\psi
    \right)
    =
    \sum_{k=0}^{m_{\boldsymbol q}}
    \left(P_{W_k^{(\ell)}}f\right)\psi
    =
    \left(P_{V_{m_{\boldsymbol q}}^{(\ell)}}f\right)\psi.
\end{aligned}
\]
Here, the orthogonality eliminates every term with $\boldsymbol r_{-\ell}\neq\boldsymbol q$, and downward closedness ensures that the remaining levels are precisely $0,\ldots,m_{\boldsymbol q}$. Applying this formula to $f=D_{N_\ell}^{(\ell)}\phi$
yields
\[
\begin{aligned}
    P_{\cS_\Lambda}A_\ell v
    =
    \left(
        P_{V_{m_{\boldsymbol q}}^{(\ell)}}
        D_{N_\ell}^{(\ell)}\phi
    \right)\psi,
    \qquad
    (I-P_{\cS_\Lambda})A_\ell v
    =
    \left(
        \left(I_\ell-P_{V_{m_{\boldsymbol q}}^{(\ell)}}\right)
        D_{N_\ell}^{(\ell)}\phi
    \right)\psi.
\end{aligned}
\]
Because
$\phi\in W_{p_\ell}^{(\ell)}
\subseteq V_{m_{\boldsymbol q}}^{(\ell)}$,
Proposition~\ref{prop:1d-leakage}, applied with
$N=N_\ell$ and $m=m_{\boldsymbol q}$, gives
\[
    \norm{
        \left(I_\ell-P_{V_{m_{\boldsymbol q}}^{(\ell)}}\right)
        D_{N_\ell}^{(\ell)}\phi
    }_{L^2_{x_\ell}}^2
    =
    \left(2^{N_\ell-m_{\boldsymbol q}}-1\right)
    \norm{
        P_{V_{m_{\boldsymbol q}}^{(\ell)}}
        D_{N_\ell}^{(\ell)}\phi
    }_{L^2_{x_\ell}}^2.
\]
Multiplying both sides by
$\norm{\psi}_{L^2_{x_{-\ell}}}^2$ gives
\[
    \norm{(I-P_{\cS_\Lambda})A_\ell v}^2
    =
    \left(2^{N_\ell-m_{\boldsymbol q}}-1\right)
    \norm{P_{\cS_\Lambda}A_\ell v}^2.
\]

Next, let $v$ be an arbitrary function in
$W_{\boldsymbol p}$.
Choose an orthonormal basis
$\{\psi_{\boldsymbol q,\alpha}\}_{\alpha=1}^{d_{\boldsymbol q}}$
of $W_{\boldsymbol q}^{(-\ell)}$,
where $d_{\boldsymbol q}:=\dim W_{\boldsymbol q}^{(-\ell)}$,
and write
\[
    v(x_\ell,x_{-\ell})
    =
    \sum_{\alpha=1}^{d_{\boldsymbol q}}
    \phi_\alpha(x_\ell)\psi_{\boldsymbol q,\alpha}(x_{-\ell}),
    \qquad
    \phi_\alpha\in W_{p_\ell}^{(\ell)}.
\]
Applying the preceding projection formulas for each term gives
\[
\begin{aligned}
    P_{\cS_\Lambda}A_\ell v
    =
    \sum_{\alpha=1}^{d_{\boldsymbol q}}
    \left(
        P_{V_{m_{\boldsymbol q}}^{(\ell)}}
        D_{N_\ell}^{(\ell)}\phi_\alpha
    \right)\psi_{\boldsymbol q,\alpha},
    \qquad
    (I-P_{\cS_\Lambda})A_\ell v
    =
    \sum_{\alpha=1}^{d_{\boldsymbol q}}
    \left(
        \left(I_\ell-P_{V_{m_{\boldsymbol q}}^{(\ell)}}\right)
        D_{N_\ell}^{(\ell)}\phi_\alpha
    \right)\psi_{\boldsymbol q,\alpha}.
\end{aligned}
\]
Orthonormality eliminates the cross terms when taking squared norms. Therefore,
\[
\begin{aligned}
    \norm{(I-P_{\cS_\Lambda})A_\ell v}^2
    &=
    \sum_{\alpha=1}^{d_{\boldsymbol q}}
    \norm{
        \left(I_\ell-P_{V_{m_{\boldsymbol q}}^{(\ell)}}\right)
        D_{N_\ell}^{(\ell)}\phi_\alpha
    }_{L^2_{x_\ell}}^2
    \\
    &=
    \left(2^{N_\ell-m_{\boldsymbol q}}-1\right)
    \sum_{\alpha=1}^{d_{\boldsymbol q}}
    \norm{
        P_{V_{m_{\boldsymbol q}}^{(\ell)}}
        D_{N_\ell}^{(\ell)}\phi_\alpha
    }_{L^2_{x_\ell}}^2
    \\
    &=
    \left(2^{N_\ell-m_{\boldsymbol q}}-1\right)
    \norm{P_{\cS_\Lambda}A_\ell v}^2.
\end{aligned}
\]

\underline{Step 2: Functions with a fixed transverse index.}
In this step, we consider the entire slice
\[
    v\in
    V_{m_{\boldsymbol q}}^{(\ell)}
    \otimes W_{\boldsymbol q}^{(-\ell)}
    =
    \bigoplus_{k=0}^{m_{\boldsymbol q}}
    W_k^{(\ell)}\otimes W_{\boldsymbol q}^{(-\ell)}.
\]
Write
\[
    v=\sum_{k=0}^{m_{\boldsymbol q}}v_k,
    \qquad
    v_k\in W_k^{(\ell)}\otimes W_{\boldsymbol q}^{(-\ell)}.
\]
With the same orthonormal basis for every $k$, we expand
\[
    v_k(x_\ell,x_{-\ell})
    =
    \sum_{\alpha=1}^{d_{\boldsymbol q}}
    \phi_{k,\alpha}(x_\ell)
    \psi_{\boldsymbol q,\alpha}(x_{-\ell}),
    \qquad
    \phi_{k,\alpha}\in W_k^{(\ell)}.
\]
Collecting the coefficients of each
$\psi_{\boldsymbol q,\alpha}$ gives
\[
    v(x_\ell,x_{-\ell})
    =
    \sum_{\alpha=1}^{d_{\boldsymbol q}}
    \Phi_\alpha(x_\ell)\psi_{\boldsymbol q,\alpha}(x_{-\ell}),
    \qquad
    \Phi_\alpha
    :=
    \sum_{k=0}^{m_{\boldsymbol q}}\phi_{k,\alpha}
    \in V_{m_{\boldsymbol q}}^{(\ell)}.
\]
Proposition~\ref{prop:1d-leakage} applies to each
$\Phi_\alpha$ with $N=N_\ell$ and $m=m_{\boldsymbol q}$.
Repeating the calculation in Step~1 yields
\[
\begin{aligned}
    \norm{(I-P_{\cS_\Lambda})A_\ell v}^2
    &=
    \sum_{\alpha=1}^{d_{\boldsymbol q}}
    \norm{
        \left(I_\ell-P_{V_{m_{\boldsymbol q}}^{(\ell)}}\right)
        D_{N_\ell}^{(\ell)}\Phi_\alpha
    }_{L^2_{x_\ell}}^2
    \\
    &=
    \left(2^{N_\ell-m_{\boldsymbol q}}-1\right)
    \sum_{\alpha=1}^{d_{\boldsymbol q}}
    \norm{
        P_{V_{m_{\boldsymbol q}}^{(\ell)}}
        D_{N_\ell}^{(\ell)}\Phi_\alpha
    }_{L^2_{x_\ell}}^2
    \\
    &=
    \left(2^{N_\ell-m_{\boldsymbol q}}-1\right)
    \norm{P_{\cS_\Lambda}A_\ell v}^2.
\end{aligned}
\]
Thus, the identity holds on the entire slice
$V_{m_{\boldsymbol q}}^{(\ell)}
\otimes W_{\boldsymbol q}^{(-\ell)}$.

\underline{Step 3: Functions in the hierarchical space.} The remaining step is to combine the transverse slices. Grouping the Haar blocks gives
\begin{equation}
    \cS_\Lambda
    =
    \bigoplus_{\boldsymbol q\in\mathcal Q_\ell}
    \left(
        \bigoplus_{k=0}^{m_{\boldsymbol q}}W_k^{(\ell)}
    \right)\otimes W_{\boldsymbol q}^{(-\ell)}
    =
    \bigoplus_{\boldsymbol q\in\mathcal Q_\ell}
    V_{m_{\boldsymbol q}}^{(\ell)}
    \otimes W_{\boldsymbol q}^{(-\ell)}.\label{eq:dc-slice}
\end{equation}
For $u\in\cS_\Lambda$, write
\[
    u=\sum_{\boldsymbol q\in\mathcal Q_\ell}u^{(\boldsymbol q)},
    \qquad
    u^{(\boldsymbol q)}
    \in
    V_{m_{\boldsymbol q}}^{(\ell)}
    \otimes W_{\boldsymbol q}^{(-\ell)},
\]
and define
\[
    g_{\boldsymbol q}
    :=P_{\cS_\Lambda}A_\ell u^{(\boldsymbol q)},
    \qquad
    e_{\boldsymbol q}
    :=(I-P_{\cS_\Lambda})A_\ell u^{(\boldsymbol q)}.
\]
By Step~2,
\[
    \norm{e_{\boldsymbol q}}^2
    =
    \left(2^{N_\ell-m_{\boldsymbol q}}-1\right)
    \norm{g_{\boldsymbol q}}^2.
\]

Since $A_\ell$ leaves the transverse factors unchanged,
and the projection formulas in Step~1 preserve their indices,
\[
    g_{\boldsymbol q}
    \in
    V_{m_{\boldsymbol q}}^{(\ell)}
    \otimes W_{\boldsymbol q}^{(-\ell)},
    \qquad
    e_{\boldsymbol q}
    \in
    V_{N_\ell}^{(\ell)}
    \otimes W_{\boldsymbol q}^{(-\ell)}.
\]
Thus, $e_{\boldsymbol q}\perp e_{\boldsymbol r}$
whenever $\boldsymbol q\neq\boldsymbol r$.
Together with
\[
    (I-P_{\cS_\Lambda})A_\ell u
    =
    \sum_{\boldsymbol q\in\mathcal Q_\ell}e_{\boldsymbol q},
\]
this gives
\[
\begin{aligned}
    \norm{(I-P_{\cS_\Lambda})A_\ell u}^2
    =
    \sum_{\boldsymbol q\in\mathcal Q_\ell}
    \norm{e_{\boldsymbol q}}^2
    =
    \sum_{\boldsymbol q\in\mathcal Q_\ell}
    \left(2^{N_\ell-m_{\boldsymbol q}}-1\right)
    \norm{g_{\boldsymbol q}}^2.
\end{aligned}
\]

We next express $g_{\boldsymbol q}$ in terms of
$X_{\boldsymbol p}^{(\ell)}$.
For $\boldsymbol p\in\Lambda$ with
$\boldsymbol p_{-\ell}=\boldsymbol q$,
\[
\begin{aligned}
    X_{\boldsymbol p}^{(\ell)}
    =
    P_{W_{\boldsymbol p}}A_\ell u
    =
    \sum_{\boldsymbol r\in\mathcal Q_\ell}
    P_{W_{\boldsymbol p}}A_\ell u^{(\boldsymbol r)}
    =
    P_{W_{\boldsymbol p}}A_\ell u^{(\boldsymbol q)}.
\end{aligned}
\]
The last equality follows because $A_\ell u^{(\boldsymbol r)}\in V_{N_\ell}^{(\ell)}\otimes W_{\boldsymbol r}^{(-\ell)}$ is orthogonal to $W_{\boldsymbol p}$ when $\boldsymbol r\neq\boldsymbol q$. Consequently,
\[
\begin{aligned}
    g_{\boldsymbol q}
    =
    P_{\cS_\Lambda}A_\ell u^{(\boldsymbol q)}
    =
    \sum_{\boldsymbol p\in\Lambda}
    P_{W_{\boldsymbol p}}A_\ell u^{(\boldsymbol q)}
    =
    \sum_{\substack{
        \boldsymbol p\in\Lambda\\
        \boldsymbol p_{-\ell}=\boldsymbol q}}
    X_{\boldsymbol p}^{(\ell)}.
\end{aligned}
\]
By orthogonality of the Haar blocks,
\[
    \norm{g_{\boldsymbol q}}^2
    =
    \sum_{\substack{
        \boldsymbol p\in\Lambda\\
        \boldsymbol p_{-\ell}=\boldsymbol q}}
    \norm{X_{\boldsymbol p}^{(\ell)}}^2.
\]
Substituting this expression into the preceding energy identity yields
\[
\begin{aligned}
    \norm{(I-P_{\cS_\Lambda})A_\ell u}^2
    =
    \sum_{\boldsymbol q\in\mathcal Q_\ell}
    \left(2^{N_\ell-m_{\boldsymbol q}}-1\right)
    \sum_{\substack{
        \boldsymbol p\in\Lambda\\
        \boldsymbol p_{-\ell}=\boldsymbol q}}
    \norm{X_{\boldsymbol p}^{(\ell)}}^2
    =
    \sum_{\boldsymbol p\in\Lambda}
    \left(2^{N_\ell-K_\ell(\boldsymbol p_{-\ell})}-1\right)
    \norm{X_{\boldsymbol p}^{(\ell)}}^2.
\end{aligned}
\]
This proves \eqref{eq:dc-leakage}.

It remains to verify the zero mean property used below
in the CFL estimate. With $h_\ell=2^{-N_\ell}$ and
periodic extension,
\[
    (A_\ell u)(x_\ell,x_{-\ell})
    =
    u(x_\ell,x_{-\ell})
    -
    u(x_\ell-h_\ell,x_{-\ell}).
\]
Periodicity implies
\[
    \int_0^1
    (A_\ell u)(x_\ell,x_{-\ell})\,dx_\ell=0
\]
for almost every $x_{-\ell}$.
If $p_\ell=0$, every $w\in W_{\boldsymbol p}$
is independent of $x_\ell$, since
$W_0^{(\ell)}=\operatorname{span}\{1\}$.
Integrating first in $x_\ell$ gives
$\ip{A_\ell u}{w}=0$.
Hence
$X_{\boldsymbol p}^{(\ell)}
=P_{W_{\boldsymbol p}}A_\ell u=0$,
which proves the final claim.
\end{proof}

\subsection{A sufficient CFL condition}
\label{subsec:dc-sufficient}

The leakage identity expresses the discarded energy in terms of the
projected directional differences on each Haar block. We now use this
identity to control the mixed terms in the forward Euler energy estimate.
The directions that can contribute to a given block are determined by
its positive hierarchical levels. This observation leads to the following
definitions.

\begin{definition}[Support]
\label{def:dc-support}
For $\boldsymbol p\in\Lambda$, define its \emph{support} by
\[
    J(\boldsymbol p)
    :=\{\ell\in\{1,\ldots,d\}:p_\ell\geq1\}.
\]
\end{definition}
By Proposition~\ref{prop:dc-leakage}, $X_{\boldsymbol p}^{(\ell)}=0$ for $\ell\notin J(\boldsymbol p)$. Thus, only directions in $J(\boldsymbol p)$ enter the energy estimate on $W_{\boldsymbol p}$.

\begin{example}
\label{ex:dc-supports}
Consider the downward closed set
\[
    \Lambda
    =\{(0,0),(1,0),(2,0),(0,1),(0,2)\}
    \subseteq\mathbb N_0^2.
\]
Its supports are
\[
\begin{aligned}
    J((0,0))=\emptyset,\quad
    J((1,0))=J((2,0))=\{1\},\quad
    J((0,1))=J((0,2))=\{2\}.
\end{aligned}
\]
Although each coordinate reaches level $2$, no index has positive
levels in both coordinates. Consequently, no Haar block receives
nonzero projected differences from both directions.
\end{example}

For a general index set, several directions may contribute to the same
block. To account for all such combinations, we collect the nonempty
supports that occur in $\Lambda$.

\begin{definition}[Admissible support]
\label{def:dc-admissible-support}
A nonempty set $S\subseteq\{1,\ldots,d\}$ is called an
\emph{admissible support} if $J(\boldsymbol p)=S$ for some
$\boldsymbol p\in\Lambda$. The collection of admissible supports is
\[
    \mathscr S(\Lambda)
    :=\{J(\boldsymbol p):
        \boldsymbol p\in\Lambda\setminus\{\boldsymbol0\}\}.
\]
\end{definition}

Downward closedness provides a simple way to check admissibility.
For a nonempty $S\subseteq\{1,\ldots,d\}$, define the indicator index
\[
    \boldsymbol1_S:=\sum_{\ell\in S}\boldsymbol e_\ell,
\]
with $\boldsymbol{e}_\ell$ the $\ell$-th standard unit vector in $\mathbb R^d$.
This is the componentwise smallest index with support $S$.
If $J(\boldsymbol p)=S$, then $\boldsymbol1_S\leq\boldsymbol p$,
so $\boldsymbol1_S\in\Lambda$. Conversely,
$J(\boldsymbol1_S)=S$. Hence, for every nonempty $S$,
\[
    S\in\mathscr S(\Lambda)
    \quad\Longleftrightarrow\quad
    \boldsymbol1_S\in\Lambda.
\]

The support identifies the directions involved in a block, while the
slice levels determine the strength of the corresponding directional
terms. For a fixed support $S$, Lemma~\ref{lem:dc-monotone} shows that
these slice levels are largest at the minimal index $\boldsymbol1_S$.
We therefore define the CFL constant using the slices through this index.

\begin{definition}[Directional levels and CFL constant]
\label{def:dc-cfl-constant}
For $S\in\mathscr S(\Lambda)$ and $\ell\in S$, define
\begin{equation}
    \kappa_\ell(S)
    :=K_\ell\bigl((\boldsymbol1_S)_{-\ell}\bigr),\qquad
    \mathcal C_\Lambda(S)
    :=\sum_{\ell\in S}c_\ell\,2^{\kappa_\ell(S)},\qquad
    \mathcal C_\Lambda(\boldsymbol c)
    :=\max_{S\in\mathscr S(\Lambda)}\mathcal C_\Lambda(S).
    \label{eq:dc-constant}
\end{equation}
If $\Lambda=\{\boldsymbol0\}$, set
$\mathcal C_\Lambda(\boldsymbol c)=0$.
\end{definition}

Here $\kappa_\ell(S)$ is the largest allowed level in direction $\ell$
when the other directions in $S$ are fixed at level $1$ and those
outside $S$ are fixed at level $0$. Equivalently,
\[
    \kappa_\ell(S)
    =\max\left\{
        k\in\mathbb N:
        k\boldsymbol e_\ell
        +\sum_{m\in S\setminus\{\ell\}}\boldsymbol e_m
        \in\Lambda
    \right\}.
\]
Thus, $c_\ell2^{\kappa_\ell(S)}$ is the directional speed divided by
the finest mesh size on this slice, and $\mathcal C_\Lambda(S)$ sums
these contributions over the directions in $S$.

The maximum over supports also has an equivalent formulation in terms of individual Haar blocks. If $\Lambda\neq\{\boldsymbol0\}$, then
\begin{equation}
    \mathcal C_\Lambda(\boldsymbol c)
    =\max_{\boldsymbol p\in\Lambda\setminus\{\boldsymbol0\}}
      \sum_{\ell\in J(\boldsymbol p)}
      c_\ell\,2^{K_\ell(\boldsymbol p_{-\ell})}.
    \label{eq:dc-block-constant}
\end{equation}
Indeed, for $J(\boldsymbol p)=S$, monotonicity gives
$K_\ell(\boldsymbol p_{-\ell})\leq\kappa_\ell(S)$ for every
$\ell\in S$, with equality in all these directions when
$\boldsymbol p=\boldsymbol1_S$. Therefore, passing from blocks to
supports does not change the maximum. The support formulation is
convenient for evaluating the constant and for stating the sharpness
criterion below.

\begin{example}[CFL constant for the preceding index set]
For the set in Example~\ref{ex:dc-supports}, the only admissible
supports are $\{1\}$ and $\{2\}$. Since
\[
    \kappa_1(\{1\})=K_1(0)=2,
    \qquad
    \kappa_2(\{2\})=K_2(0)=2,
\]
we obtain
\[
    \mathcal C_\Lambda(\{1\})=4c_1,
    \qquad
    \mathcal C_\Lambda(\{2\})=4c_2,
    \qquad
    \mathcal C_\Lambda(\boldsymbol c)=4\max\{c_1,c_2\}.
\]
\end{example}

The next example shows how the constant changes when both directions
can occur in the same block.

\begin{example}[A full grid]
Let $\Lambda=\{0,1,2\}^2$. In addition to the singleton supports,
$\{1,2\}$ is now admissible, and
\[
\begin{aligned}
    \kappa_1(\{1,2\})
    &=\max\{k:(k,1)\in\Lambda\}=2,\\
    \kappa_2(\{1,2\})
    &=\max\{k:(1,k)\in\Lambda\}=2.
\end{aligned}
\]
Consequently,
\[
    \mathcal C_\Lambda(\{1,2\})=4(c_1+c_2).
\]
Since $c_1,c_2\geq0$, this is at least as large as either singleton
contribution. Hence,
\[
    \mathcal C_\Lambda(\boldsymbol c)=4(c_1+c_2).
\]
\end{example}

We now show that the reciprocal of $\mathcal C_\Lambda(\boldsymbol c)$ provides an admissible time step.
\begin{theorem}[Sufficient CFL condition on downward closed sets]
\label{thm:dc-sufficient}
If
\begin{equation}
    \Delta t\,\mathcal C_\Lambda(\boldsymbol c)\leq1,
    \label{eq:dc-cfl}
\end{equation}
then the scheme \eqref{eq:dc-scheme} is $L^2$-contractive:
\[
    \norm{u^{n+1}}\leq\norm{u^n}
    \qquad\text{for every }u^n\in\cS_\Lambda.
\]
\end{theorem}

\begin{proof}
Abbreviate $u=u^n$ and $P:=P_{\cS_\Lambda}$. For each direction $\ell$, define
\[
    g_\ell:=PA_\ell u,
    \qquad
    e_\ell:=(I-P)A_\ell u.
\]
For $\boldsymbol p\in\Lambda$, also set
\[
    a_{\boldsymbol p,\ell}
    :=\norm{X_{\boldsymbol p}^{(\ell)}},
    \qquad
    \omega_\ell(\boldsymbol p)
    :=2^{N_\ell-K_\ell(\boldsymbol p_{-\ell})}-1.
\]
The proof proceeds by decomposing the energy change into Haar blocks,
estimating the directional coupling within each block, and then using
the support to obtain a common CFL condition.

\underline{Step 1: Decompose the energy change into Haar blocks.}
Since $u\in\cS_\Lambda$ and $P$ is the orthogonal projection, $\ip{u}{g_\ell}=\ip{u}{A_\ell u}$. Applying Lemma~\ref{lem:D-energy} in direction $\ell$ at level $N_\ell$ gives
\[
    2\ip{u}{g_\ell}
    =2\ip{u}{A_\ell u}
    =\norm{A_\ell u}^2
    =\norm{g_\ell}^2+\norm{e_\ell}^2.
\]
Using this identity in the squared norm of
$u^{n+1}=u-\sum_{\ell=1}^d r_\ell g_\ell$, we obtain
\begin{equation}
\begin{aligned}
    \norm{u^{n+1}}^2
    &=\norm{u}^2
      -2\sum_{\ell=1}^d r_\ell\ip{u}{g_\ell}
      +\sum_{\ell=1}^d r_\ell^2\norm{g_\ell}^2
      +2\sum_{1\leq\ell<m\leq d}
        r_\ell r_m\ip{g_\ell}{g_m}\\
    &=\norm{u}^2
      -\sum_{\ell=1}^d r_\ell(1-r_\ell)\norm{g_\ell}^2
      -\sum_{\ell=1}^d r_\ell\norm{e_\ell}^2
      +2\sum_{1\leq\ell<m\leq d}
        r_\ell r_m\ip{g_\ell}{g_m}.
\end{aligned}
\label{eq:dc-energy-expansion}
\end{equation}

Each term can now be expressed in the same Haar decomposition.
Orthogonality gives
\[
    \norm{g_\ell}^2
    =\sum_{\boldsymbol p\in\Lambda}a_{\boldsymbol p,\ell}^2,
    \qquad
    \ip{g_\ell}{g_m}
    =\sum_{\boldsymbol p\in\Lambda}
      \ip{X_{\boldsymbol p}^{(\ell)}}
         {X_{\boldsymbol p}^{(m)}}.
\]
Moreover, Proposition~\ref{prop:dc-leakage} yields
\[
    \norm{e_\ell}^2
    =\sum_{\boldsymbol p\in\Lambda}
      \omega_\ell(\boldsymbol p)a_{\boldsymbol p,\ell}^2.
\]
Substituting these formulas into \eqref{eq:dc-energy-expansion} and
grouping the contributions from each block, we find
\[
    \norm{u^{n+1}}^2
    =\norm{u}^2-\sum_{\boldsymbol p\in\Lambda}Q_{\boldsymbol p},
\]
where
\[
    Q_{\boldsymbol p}
    := \sum_{\ell\in J(\boldsymbol p)}
          r_\ell\bigl(1-r_\ell+\omega_\ell(\boldsymbol p)\bigr)
          a_{\boldsymbol p,\ell}^2
       -2\sum_{\substack{\ell<m\\\ell,m\in J(\boldsymbol p)}}
          r_\ell r_m
          \ip{X_{\boldsymbol p}^{(\ell)}}
             {X_{\boldsymbol p}^{(m)}}.
\]
Only directions in $J(\boldsymbol p)$ occur because the other
projected components vanish. In particular, $Q_{\boldsymbol0}=0$.
It remains to prove $Q_{\boldsymbol p}\geq0$ for each nonzero index.

\underline{Step 2: Control the mixed terms within a block.}

Fix $\boldsymbol p\neq\boldsymbol0$ and abbreviate
\[
    J:=J(\boldsymbol p),
    \qquad
    a_\ell:=a_{\boldsymbol p,\ell},
    \qquad
    \omega_\ell:=\omega_\ell(\boldsymbol p).
\]
Cauchy-Schwarz gives
$\ip{X_{\boldsymbol p}^{(\ell)}}{X_{\boldsymbol p}^{(m)}}
\leq a_\ell a_m$. Since $r_\ell r_m\geq0$, this upper bound on
each mixed inner product gives the following lower bound on
$Q_{\boldsymbol p}$:
\[
    Q_{\boldsymbol p}
    \geq
      \sum_{\ell\in J}r_\ell(1-r_\ell+\omega_\ell)a_\ell^2
      -2\sum_{\substack{\ell<m\\\ell,m\in J}}
        r_\ell r_m a_\ell a_m
    =\sum_{\ell\in J}r_\ell(1+\omega_\ell)a_\ell^2
      -\left(\sum_{\ell\in J}r_\ell a_\ell\right)^2.
\]
The diagonal quadratic terms and the mixed terms have thus combined
into a single square. To compare this square with the remaining
diagonal sum, we use the weights supplied by the leakage identity:
\[
    \sum_{\ell\in J}r_\ell a_\ell
    =\sum_{\ell\in J}
      \sqrt{\frac{r_\ell}{1+\omega_\ell}}\,
      \left(\sqrt{r_\ell(1+\omega_\ell)}\,a_\ell\right).
\]
As $1+\omega_\ell>0$, Cauchy-Schwarz yields
\[
    \left(\sum_{\ell\in J}r_\ell a_\ell\right)^2
    \leq
    \left(\sum_{\ell\in J}\frac{r_\ell}{1+\omega_\ell}\right)
    \left(\sum_{\ell\in J}r_\ell(1+\omega_\ell)a_\ell^2\right).
\]
Consequently,
\[
    Q_{\boldsymbol p}
    \geq
    \left(1-\sum_{\ell\in J}\frac{r_\ell}{1+\omega_\ell}\right)
    \sum_{\ell\in J}r_\ell(1+\omega_\ell)a_\ell^2.
\]

The coefficient in parentheses is where the directional slice levels enter the time-step restriction.
\[
    \frac{r_\ell}{1+\omega_\ell}
    =\frac{c_\ell\Delta t\,2^{N_\ell}}
           {2^{N_\ell-K_\ell(\boldsymbol p_{-\ell})}}
    =\Delta t\,c_\ell\,2^{K_\ell(\boldsymbol p_{-\ell})}.
\]
We therefore obtain
\begin{equation}
    Q_{\boldsymbol p}
    \geq
    \left(
        1-\Delta t\sum_{\ell\in J(\boldsymbol p)}
          c_\ell\,2^{K_\ell(\boldsymbol p_{-\ell})}
    \right)
    \sum_{\ell\in J}r_\ell(1+\omega_\ell)a_\ell^2.
    \label{eq:dc-block-positivity}
\end{equation}
The final sum is nonnegative, so it suffices to bound the directional
sum in parentheses uniformly over the blocks.

\underline{Step 3: Obtain a common bound from the support.}

Let $S:=J(\boldsymbol p)$. Since $\boldsymbol p\neq\boldsymbol0$,
the support is nonempty and $\boldsymbol1_S\leq\boldsymbol p$.
Downward closedness gives $\boldsymbol1_S\in\Lambda$, and therefore
$S\in\mathscr S(\Lambda)$. For every $\ell\in S$,
\[
    (\boldsymbol1_S)_{-\ell}\leq\boldsymbol p_{-\ell}.
\]
By Lemma~\ref{lem:dc-monotone},
\[
    K_\ell(\boldsymbol p_{-\ell})
    \leq K_\ell\bigl((\boldsymbol1_S)_{-\ell}\bigr)
    =\kappa_\ell(S).
\]
Using $c_\ell\geq0$ and the definition of the CFL constant, we conclude
\[
    \sum_{\ell\in J(\boldsymbol p)}
    c_\ell\,2^{K_\ell(\boldsymbol p_{-\ell})}
    \leq\sum_{\ell\in S}c_\ell\,2^{\kappa_\ell(S)}
    =\mathcal C_\Lambda(S)
    \leq\mathcal C_\Lambda(\boldsymbol c).
\]
Under \eqref{eq:dc-cfl}, the first factor on the right-hand side of \eqref{eq:dc-block-positivity} is therefore nonnegative. Hence $Q_{\boldsymbol p}\geq0$ for every nonzero index as well. Together with $Q_{\boldsymbol0}=0$, this gives
\[
    \norm{u^{n+1}}^2
    =\norm{u^n}^2-\sum_{\boldsymbol p\in\Lambda}Q_{\boldsymbol p}
    \leq\norm{u^n}^2.
\]
\end{proof}

\begin{remark}[Interpretation]
\label{rem:dc-interpretation}
The quantity $2^{-K_\ell(\boldsymbol q)}$ is the finest mesh size in
direction $\ell$ on the transverse slice $\boldsymbol q$. By
\eqref{eq:dc-block-constant}, condition \eqref{eq:dc-cfl} is equivalent
to requiring
\[
    \Delta t\sum_{\ell\in J(\boldsymbol p)}
    c_\ell\,2^{K_\ell(\boldsymbol p_{-\ell})}\leq1
    \qquad
    \text{for every }\boldsymbol p\in\Lambda\setminus\{\boldsymbol0\}.
\]
Thus, each block satisfies a CFL condition involving the mesh sizes
on its directional slices. The factors associated with the ambient
levels $N_\ell$ cancel in Step~2, leaving a bound determined by the
transport velocities and the slice levels of $\Lambda$.

For $\Lambda=\Lambda_N$, the weighted Cauchy--Schwarz estimate in
Step~2 provides an alternative to the counting estimate in
Proposition~\ref{prop:d-mixed-terms}. Both arguments recover the same
threshold for $\Lambda_N$. The weighted estimate also recovers the
sharp CFL condition for full tensor-product grids, including
anisotropic ones, as shown below.
\end{remark}

\subsection{A sharpness criterion}
\label{subsec:dc-sharpness}

The sufficient CFL condition becomes necessary if the directional
contributions associated with a maximizing support can be realized
simultaneously by a single mode in $\cS_\Lambda$. To identify when this
is possible, for $S\in\mathscr S(\Lambda)$ define its \emph{corner} by
\begin{equation}
    \boldsymbol\kappa(S)
    :=
    \sum_{\ell\in S}\kappa_\ell(S)\,\boldsymbol e_\ell
    \in\mathbb N_0^d.
    \label{eq:dc-corner}
\end{equation}
Each component $\kappa_\ell(S)$ is an admissible directional level
when the other active coordinates are fixed at level $1$. Their
simultaneous realization is a stronger requirement: the vector
$\boldsymbol\kappa(S)$ need not belong to $\Lambda$. When it does belong
to $\Lambda$ for a maximizing support, it determines an alternating
mode that attains the sufficient CFL bound.

\begin{theorem}[Sharpness criterion]
\label{thm:dc-sharpness}
Let $S^*\in\mathscr S(\Lambda)$ be a maximizer of
$\mathcal C_\Lambda(\cdot)$ in \eqref{eq:dc-constant}. If
$\boldsymbol\kappa(S^*)\in\Lambda$, then the scheme
\eqref{eq:dc-scheme} is $L^2$-contractive for every
$u^n\in\cS_\Lambda$ if and only if
$\Delta t\,\mathcal C_\Lambda(\boldsymbol c)\leq1$.
In particular, this holds whenever some singleton $S^*=\{\ell\}$
is a maximizer, since then
$\boldsymbol\kappa(S^*)=N_\ell\boldsymbol e_\ell\in\Lambda$.
\end{theorem}

\begin{proof}
Sufficiency follows from Theorem~\ref{thm:dc-sufficient}. To prove
necessity, write $S=S^*$ and
$\boldsymbol\kappa=\boldsymbol\kappa(S^*)$. We first identify the
directional slice levels at this corner and then construct an
alternating mode supported on $W_{\boldsymbol\kappa}$.

\underline{Step 1: The slice levels at the corner.}

We claim that
$K_\ell(\boldsymbol\kappa_{-\ell})=\kappa_\ell$ for every
$\ell\in S$. For $m\in S$, admissibility gives
$(1;(\boldsymbol1_S)_{-m})_m=\boldsymbol1_S\in\Lambda$, and hence
$\kappa_m=\kappa_m(S)\geq1$. Therefore
$\boldsymbol\kappa_{-\ell}\geq(\boldsymbol1_S)_{-\ell}$
componentwise. Lemma~\ref{lem:dc-monotone} now yields
\[
    K_\ell(\boldsymbol\kappa_{-\ell})
    \leq
    K_\ell\bigl((\boldsymbol1_S)_{-\ell}\bigr)
    =\kappa_\ell(S)=\kappa_\ell.
\]
On the other hand, $\boldsymbol\kappa\in\Lambda$ implies
$(\kappa_\ell;\boldsymbol\kappa_{-\ell})_\ell\in\Lambda$, so
$K_\ell(\boldsymbol\kappa_{-\ell})\geq\kappa_\ell$.
The two inequalities prove the claim.

\underline{Step 2: An extremal mode.}

For $k\geq1$, let $\psi_k^{(\ell)}\in W_k^{(\ell)}$ be the
level-$k$ alternating function defined by
$\psi_k^{(\ell)}\big|_{I_{k,j}}=(-1)^j$, and set
$\psi_0^{(\ell)}:=1$. Consider
\[
    v:=\bigotimes_{\ell=1}^d\psi_{\kappa_\ell}^{(\ell)}
    \in W_{\boldsymbol\kappa}\subset\cS_\Lambda.
\]
If $m\notin S$, then $v$ is constant in $x_m$, so $A_mv=0$.
For $\ell\in S$, the transverse factor of $A_\ell v$ is
\[
    \bigotimes_{m\neq\ell}\psi_{\kappa_m}^{(m)}
    \in W_{\boldsymbol\kappa_{-\ell}}^{(-\ell)}.
\]
By the slice decomposition \eqref{eq:dc-slice} and Step~1,
$P_{\cS_\Lambda}$ therefore acts on $A_\ell v$ as
$P_{V_{\kappa_\ell}^{(\ell)}}$ in the $\ell$th coordinate and as
the identity in every other coordinate. Lemma~\ref{lem:multilevel},
together with
$D_{\kappa_\ell}\psi_{\kappa_\ell}^{(\ell)}
=2\psi_{\kappa_\ell}^{(\ell)}$ for $\kappa_\ell\geq1$, gives
\begin{align*}
    P_{\cS_\Lambda}A_\ell v
    =
    2^{\kappa_\ell-N_\ell}
    \left(\bigotimes_{m<\ell}\psi_{\kappa_m}^{(m)}\right)
    \otimes
    \bigl(D_{\kappa_\ell}\psi_{\kappa_\ell}^{(\ell)}\bigr)
    \otimes
    \left(\bigotimes_{m>\ell}\psi_{\kappa_m}^{(m)}\right)
    =
    2^{\kappa_\ell-N_\ell+1}v.
\end{align*}
Thus $v$ is a common eigenfunction of the projected directional
operators. Taking $u^n=v$ in \eqref{eq:dc-scheme}, we obtain
\begin{align*}
    u^{n+1}
    =
    \left(1-\sum_{\ell\in S}
        r_\ell\,2^{\kappa_\ell-N_\ell+1}\right)v
    =
    \bigl(1-2\Delta t\,\mathcal C_\Lambda(S^*)\bigr)v
    =
    \bigl(1-2\Delta t\,\mathcal C_\Lambda(\boldsymbol c)\bigr)v.
\end{align*}
Contractivity for this nonzero mode requires
$\abs{1-2\Delta t\,\mathcal C_\Lambda(\boldsymbol c)}\leq1$.
Since $\Delta t\,\mathcal C_\Lambda(\boldsymbol c)\geq0$, this is
equivalent to
$\Delta t\,\mathcal C_\Lambda(\boldsymbol c)\leq1$.
\end{proof}

The criterion recovers the familiar full grid condition directly. For a full grid (i.e., tensor product) index set, all directional maximal levels can occur in the same block, and the support containing every coordinate is a maximizer.

\begin{corollary}[Anisotropic full grids]
\label{cor:dc-fullgrid}
Let
$\Lambda=\{\boldsymbol p\in\mathbb N_0^d:
p_\ell\leq N_\ell,\ 1\leq\ell\leq d\}$
with $N_\ell\geq1$. Then the scheme \eqref{eq:dc-scheme} is
$L^2$-contractive if and only if
\[
    \Delta t\sum_{\ell=1}^d\frac{c_\ell}{h_\ell}\leq1,
    \qquad h_\ell=2^{-N_\ell}.
\]
\end{corollary}

\begin{proof}
Here $K_\ell(\boldsymbol q)=N_\ell$ for every
$\boldsymbol q\in\mathcal Q_\ell$. Consequently,
$\mathcal C_\Lambda(S)=\sum_{\ell\in S}c_\ell2^{N_\ell}$ is
maximized by $S^*=\{1,\ldots,d\}$, with
\[
    \mathcal C_\Lambda(\boldsymbol c)
    =\sum_{\ell=1}^d c_\ell2^{N_\ell},
    \qquad
    \boldsymbol\kappa(S^*)=(N_1,\ldots,N_d)\in\Lambda.
\]
The result follows from Theorem~\ref{thm:dc-sharpness}.
\end{proof}

For the standard sparse grid index set $\Lambda_N$, increasing the
size of the support reduces the available level in each active
direction. The resulting CFL constant is attained by a singleton,
which gives another proof of the sharp condition from
Section~\ref{sec:proof-arbitrary-d}.

\begin{corollary}[An alternative proof of
Theorem~\ref{thm:d-hypercube}]
\label{cor:dc-simplex}
Let $\Lambda=\Lambda_N$ with $N\geq1$. Then
$\mathcal C_\Lambda(\boldsymbol c)
=2^N\max_{1\leq\ell\leq d}c_\ell$, attained by a singleton,
and the scheme is $L^2$-contractive if and only if
\[
    \Delta t\,2^N\max_{1\leq\ell\leq d}c_\ell\leq1.
\]
When $\max_\ell c_\ell>0$, this is equivalent to
$\Delta t\leq h/\max_\ell c_\ell$ with $h=2^{-N}$.
If all velocities vanish, the scheme is contractive for every
$\Delta t\geq0$.
\end{corollary}

\begin{proof}
The admissible supports are precisely the nonempty sets
$S\subseteq\{1,\ldots,d\}$ with $\#S\leq N$.
For such an $S$ and every $\ell\in S$, we have
$\kappa_\ell(S)=N-\#S+1$. Hence
\begin{align*}
    \mathcal C_\Lambda(S)
    &=2^{N-\#S+1}\sum_{\ell\in S}c_\ell
    \\
    &\leq2^{N-\#S+1}\,\#S\,\max_\ell c_\ell
    \leq2^N\max_\ell c_\ell
    =\mathcal C_\Lambda(\{\ell^*\}),
\end{align*}
where $c_{\ell^*}=\max_\ell c_\ell$ and the second inequality uses
$\#S\,2^{1-\#S}\leq1$. Thus the singleton $\{\ell^*\}$ is a
maximizer, and Theorem~\ref{thm:dc-sharpness} applies.
\end{proof}

For a general downward closed set, the corner of a maximizing support may lie outside $\Lambda$. In that case, the criterion in Theorem~\ref{thm:dc-sharpness} does not decide whether the sufficient bound is sharp. The following example gives numerical evidence that the exact threshold can be strictly larger.

\begin{example}[An L-shaped set: the sufficient condition need not be sharp]
\label{ex:dc-Lshape}
Let $d=2$ and consider the L-shaped index set
\[
    \Lambda
    =\{\boldsymbol p\in\mathbb N_0^2:\boldsymbol p\leq(4,1)\}
    \cup
    \{\boldsymbol p\in\mathbb N_0^2:\boldsymbol p\leq(1,3)\},
\]
where the inequalities are componentwise. Then $N_1=4$, $N_2=3$,
and
\[
    \kappa_1(\{1,2\})=K_1(1)=4,
    \qquad
    \kappa_2(\{1,2\})=K_2(1)=3.
\]
It follows that
\[
    \mathcal C_\Lambda(\{1\})=2^4c_1,
    \qquad
    \mathcal C_\Lambda(\{2\})=2^3c_2,
    \qquad
    \mathcal C_\Lambda(\{1,2\})=2^4c_1+2^3c_2.
\]
For $c_1,c_2>0$, the unique maximizer is $S^*=\{1,2\}$, but its
corner $\boldsymbol\kappa(S^*)=(4,3)$ does not belong to $\Lambda$.
Theorem~\ref{thm:dc-sharpness} therefore does not apply.

The computations in Section~\ref{subsec:numerics-dc} indicate that
the sufficient bound is not sharp for the two velocity vectors
considered there. For $\boldsymbol c=(1,1)$,
Theorem~\ref{thm:dc-sufficient} gives $\Delta t\leq1/24$, whereas the
numerically computed contractivity threshold is
$\Delta t^*\approx5.536\times10^{-2}$. For $\boldsymbol c=(2,5)$,
the sufficient bound is $\Delta t\leq1/72$, whereas
$\Delta t^*\approx2.148\times10^{-2}$.
Thus the corner criterion provides a sufficient geometric condition
for sharpness; a complete geometric characterization of the threshold
for general downward closed sets requires further analysis.
\end{example}

\section{Numerical examples}\label{sec:numerics}

In this part, we present numerical examples to validate the sharp CFL condition for the sparse grid scheme. In Sections~\ref{subsec:numerics-2d} and \ref{subsec:numerics-higher-d}, we study the sparse grid method for transport equations in 2D and 4D, respectively. We take the maximum mesh level to be $N=5$ in both cases. We show that the CFL condition is sharp numerically.
In Section~\ref{subsec:numerics-dc}, we check a downward closed index sets of $L$-shaped form and show that the theoretical CFL condition is smaller than the true threshold numerically.

\subsection{Two dimensional example}
\label{subsec:numerics-2d}

We numerically examine the sharp CFL condition in
Theorem~\ref{thm:2d-square} and the corresponding spectral characterization in Corollary~\ref{cor:2d-amplification-radius} for the periodic transport problem in 2D:
\begin{equation*}
    u_t+c_1u_{x}+c_2u_{y}=0, \qquad (x,y)\in[0,1]^2.
\end{equation*}
We define the CFL number by $\nu:={\Delta t}/{h}$.

For both isotropic velocity $\boldsymbol c=(1,1)$ and anisotropic velocity $\boldsymbol c=(2,5)$, the computed norm agrees with $\norm{G_{\boldsymbol c}(\nu)}_2=\max\{1,2m\nu-1\}$ with $m=\max(c_1,c_2)$, as predicted by \eqref{eq:2d-sparse-amplification-radius}. It remains equal to one through the predicted threshold and increases immediately once that threshold is exceeded.

\begin{figure}[htbp]
    \centering
    \begin{subfigure}[t]{0.485\textwidth}
        \centering
        \includegraphics[width=\linewidth]{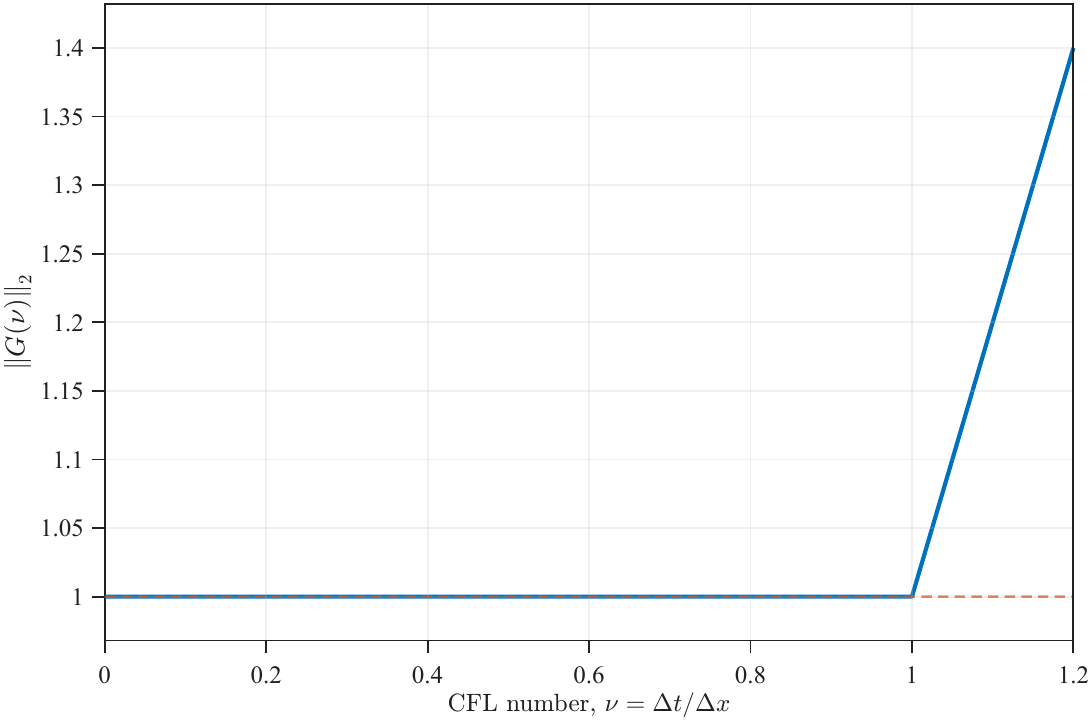}
        \caption{Isotropic velocity $\boldsymbol c=(1,1)$.}
        \label{fig:fe-norm-scan-d2-isotropic}
    \end{subfigure}
    \hfill
    \begin{subfigure}[t]{0.485\textwidth}
        \centering
        \includegraphics[width=\linewidth]{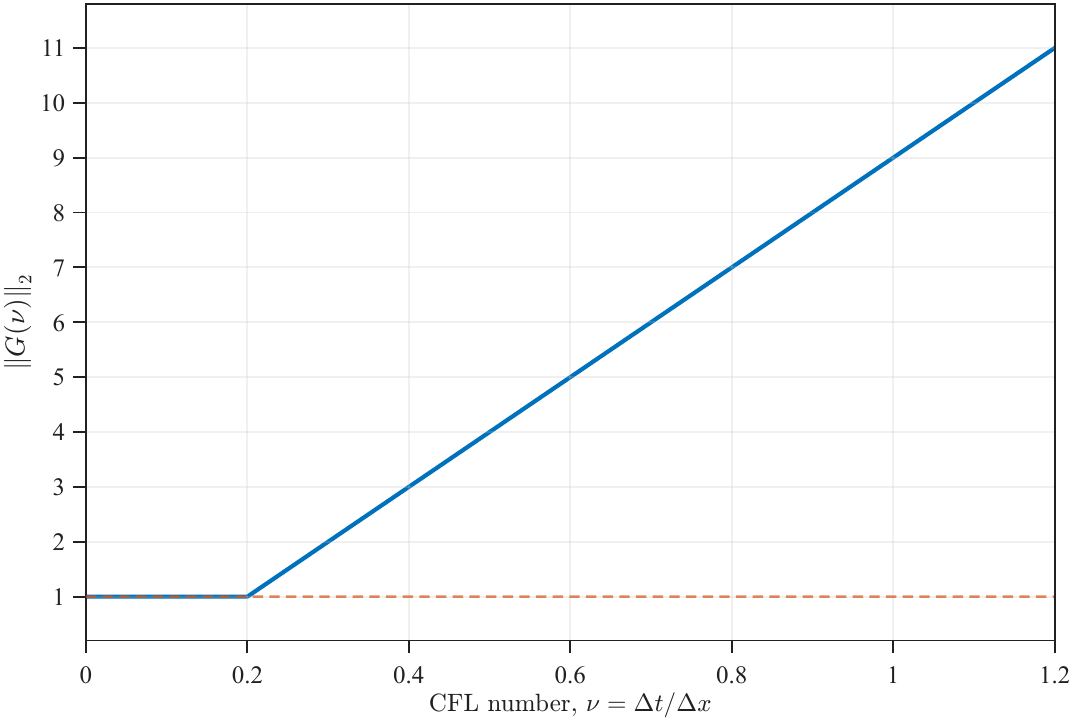}
        \caption{Anisotropic velocity $\boldsymbol c=(2,5)$.}
        \label{fig:fe-norm-scan-d2-anisotropic}
    \end{subfigure}
    \caption{Forward Euler amplification norms for the two dimensional sparse grid scheme. The computed curves remain equal to $1$ through $\nu=1$ for $\boldsymbol c=(1,1)$ and through $\nu=1/5$ for $\boldsymbol c=(2,5)$, then grow linearly beyond the
    respective sharp thresholds.}
    \label{fig:fe-norm-scan-d2}
\end{figure}

We then take the initial condition
\begin{equation*}
    u_0(x,y) = \cos(2\pi(x+y)),
\end{equation*}
and apply the forward Euler scheme \eqref{eq:sparse-update} and monitor the normalized discrete $L^2$ norm $E^n:={\norm{u_h^n}_{L^2}}/{\norm{u_h^0}_{L^2}}$ for $t_n=n\Delta t$.
We first consider the isotropic velocity $\boldsymbol c=(1,1)$. In this case, at the sharp CFL threshold $\nu=1$, the $L^2$ norm of the solution decays with respect to time. Increasing the CFL number to $\nu=1.001$ produces the exponential growth after a short transient, as shown in Figure~\ref{fig:l2-history-c1-1-c2-1}. The numerical results confirm the sharp CFL condition $\nu=1$ for the sparse grid scheme. We next consider the anisotropic velocity $\boldsymbol c=(2,5)$. At the CFL threshold $\nu=0.2$, the scheme is stable, while at $\nu=0.201$, the scheme becomes unstable. The results are shown in Figure~\ref{fig:l2-history-c1-2-c2-5}.

\begin{figure}[htbp]
    \centering
    \begin{subfigure}[t]{0.485\textwidth}
        \centering
        \includegraphics[width=\linewidth]{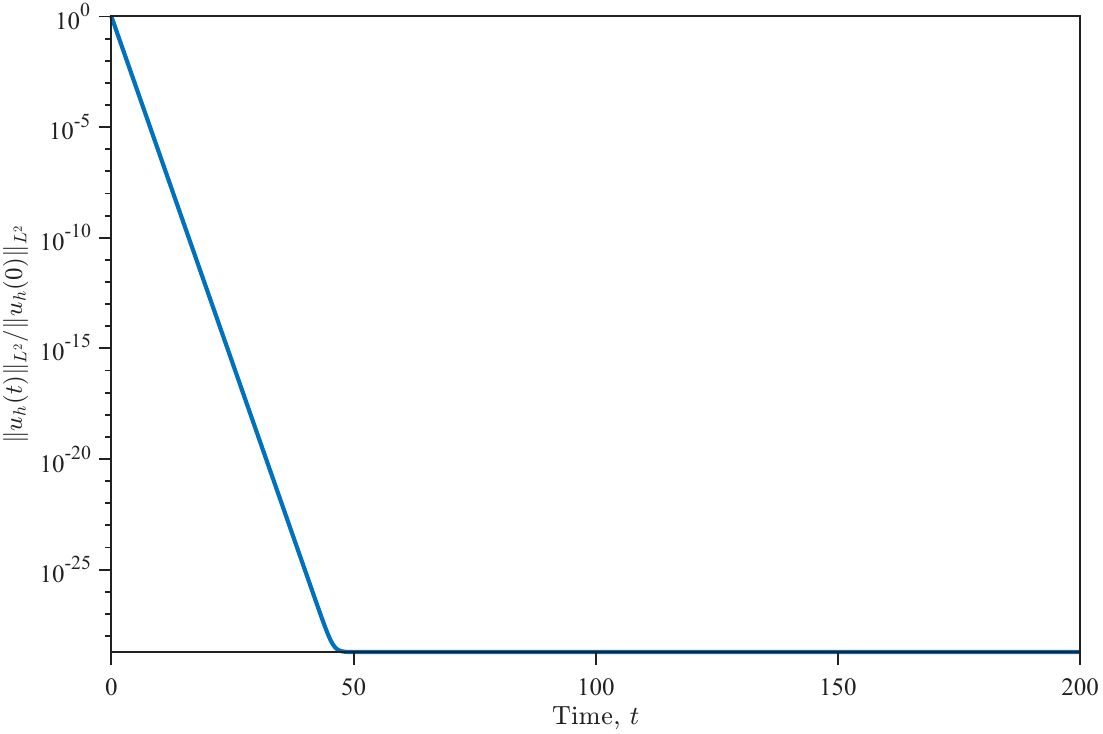}
        \caption{Sharp CFL condition, $\nu=\Delta t/h=1$.}
        \label{fig:l2-c1-1-c2-1-cfl1000}
    \end{subfigure}
    \hfill
    \begin{subfigure}[t]{0.485\textwidth}
        \centering
        \includegraphics[width=\linewidth]{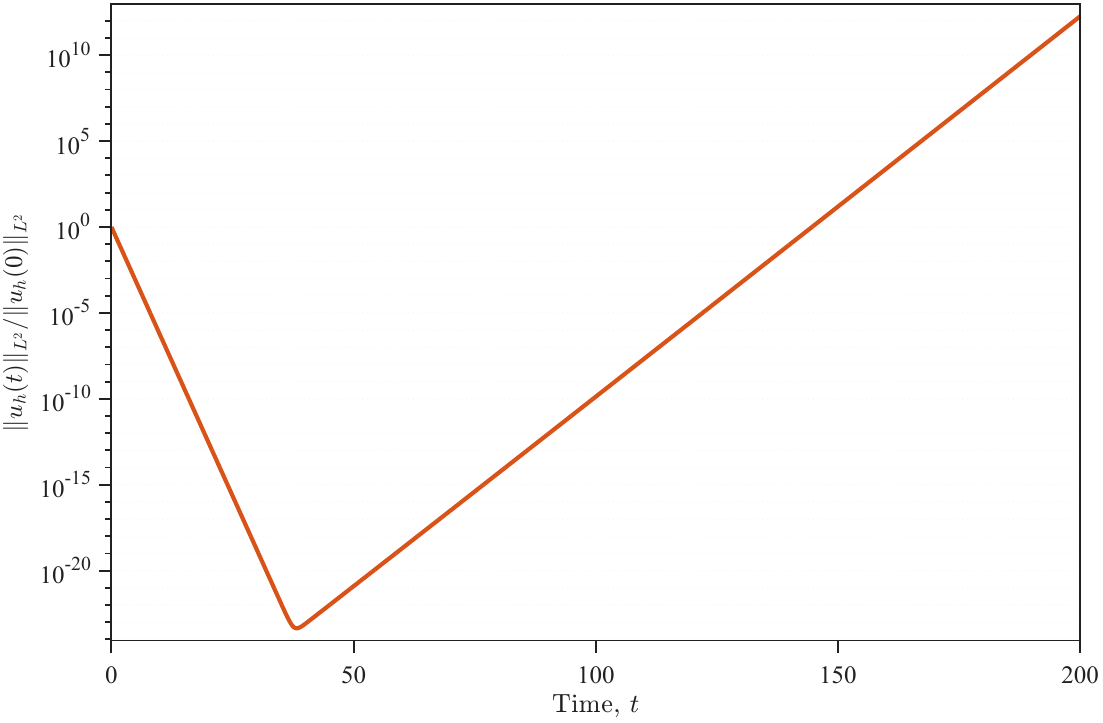}
        \caption{Above the sharp CFL condition, $\nu=\Delta t/h=1.001$.}
        \label{fig:l2-c1-1-c2-1-cfl1001}
    \end{subfigure}
    \caption{The time evolution of the normalized $L^2$ norm for the isotropic velocity $(c_1,c_2)=(1,1)$. At the predicted sharp threshold $\nu=1$ (left), the norm never exceeds its initial value. Increasing the CFL number by $10^{-3}$ (right) results in exponential growth of the norm after a short transient.}
    \label{fig:l2-history-c1-1-c2-1}
\end{figure}

\begin{figure}[htbp]
    \centering
    \begin{subfigure}[t]{0.485\textwidth}
        \centering
        \includegraphics[width=\linewidth]{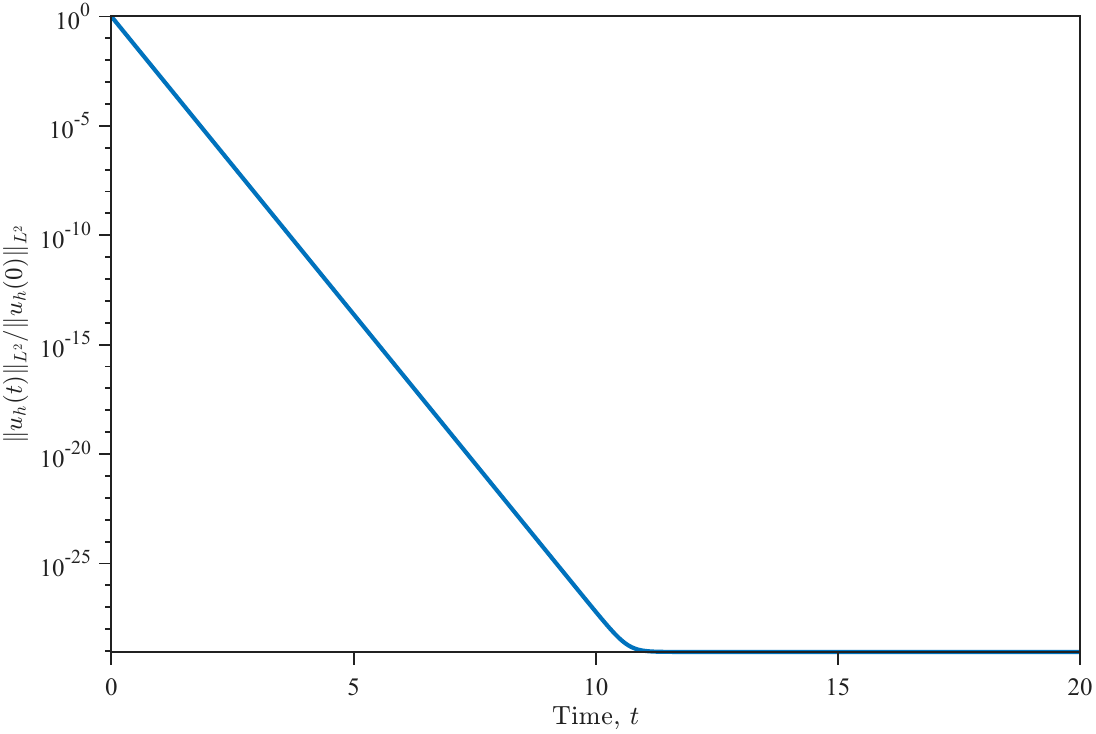}
        \caption{Sharp endpoint, $\nu=0.200$.}
        \label{fig:l2-c1-2-c2-5-cfl0200}
    \end{subfigure}
    \hfill
    \begin{subfigure}[t]{0.485\textwidth}
        \centering
        \includegraphics[width=\linewidth]{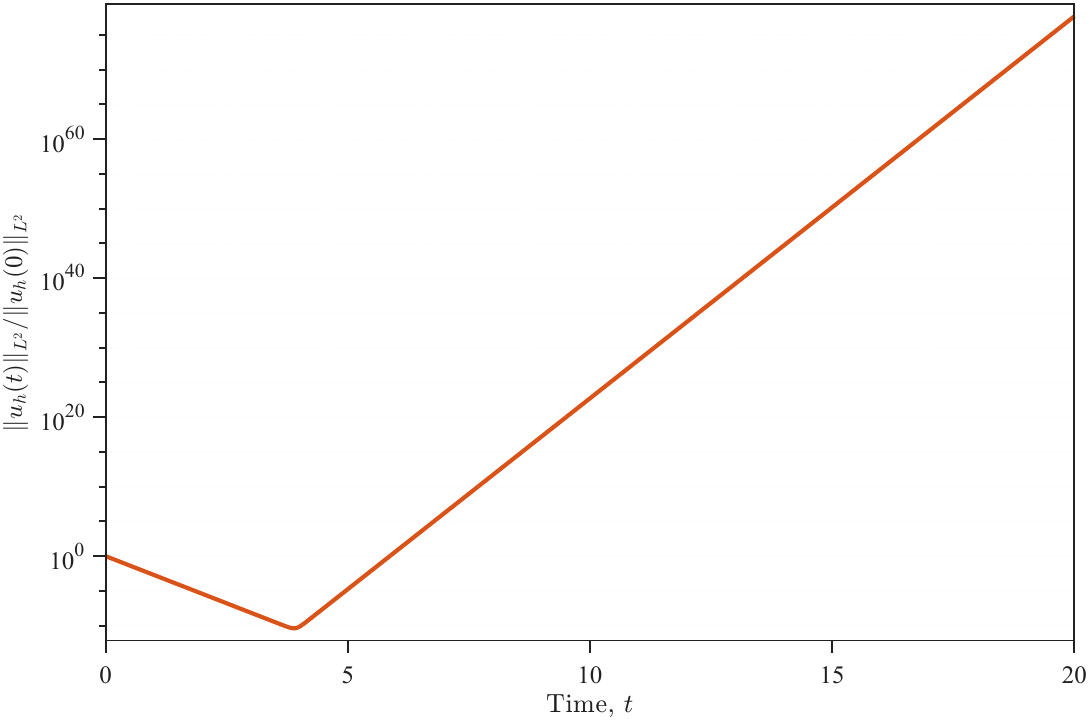}
        \caption{Above the endpoint, $\nu=0.201$.}
        \label{fig:l2-c1-2-c2-5-cfl0201}
    \end{subfigure}
    \caption{The time evolution of the normalized $L^2$ norm for the anisotropic velocity $(c_1,c_2)=(2,5)$. At the predicted sharp threshold $\nu=0.2$ (left), the norm never exceeds its initial value. Increasing the CFL number by $10^{-3}$ (right) results in exponential growth of the norm after a short transient.}
    \label{fig:l2-history-c1-2-c2-5}
\end{figure}

\FloatBarrier

\subsection{Higher dimensional example}
\label{subsec:numerics-higher-d}

In this part, we present numerical results for the sparse grid scheme in higher dimensions. We consider the periodic transport problem in four dimensions:
\begin{equation*}
    u_t+\sum_{\ell=1}^4 c_\ell u_{x_\ell}=0,
    \qquad
    \boldsymbol x=(x_1,x_2,x_3,x_4)\in[0,1]^4.
\end{equation*}
with $\nu:=\Delta t/h$. Two velocity vectors are used for comparison. For $\boldsymbol c=(1,1,1,1)$, all four directions reach the endpoint at the
same value of $\nu$; for $\boldsymbol c=(1,2,3,4)$, the fourth direction is
the unique limiting one. Figure~\ref{fig:fe-norm-scan-d4} shows the resulting
breakpoints at $\nu=1$ and $\nu=1/4$, respectively. In both cases the
computed norm follows
$\norm{G_{\boldsymbol c}(\nu)}_2=\max\{1,2m\nu-1\}$, with
$m=\max_{1\leq\ell\leq4}c_\ell$, in agreement with
\eqref{eq:d-exact-amplification-radius}.

\begin{figure}[htbp]
    \centering
    \begin{subfigure}[t]{0.485\textwidth}
        \centering
        \includegraphics[width=\linewidth]{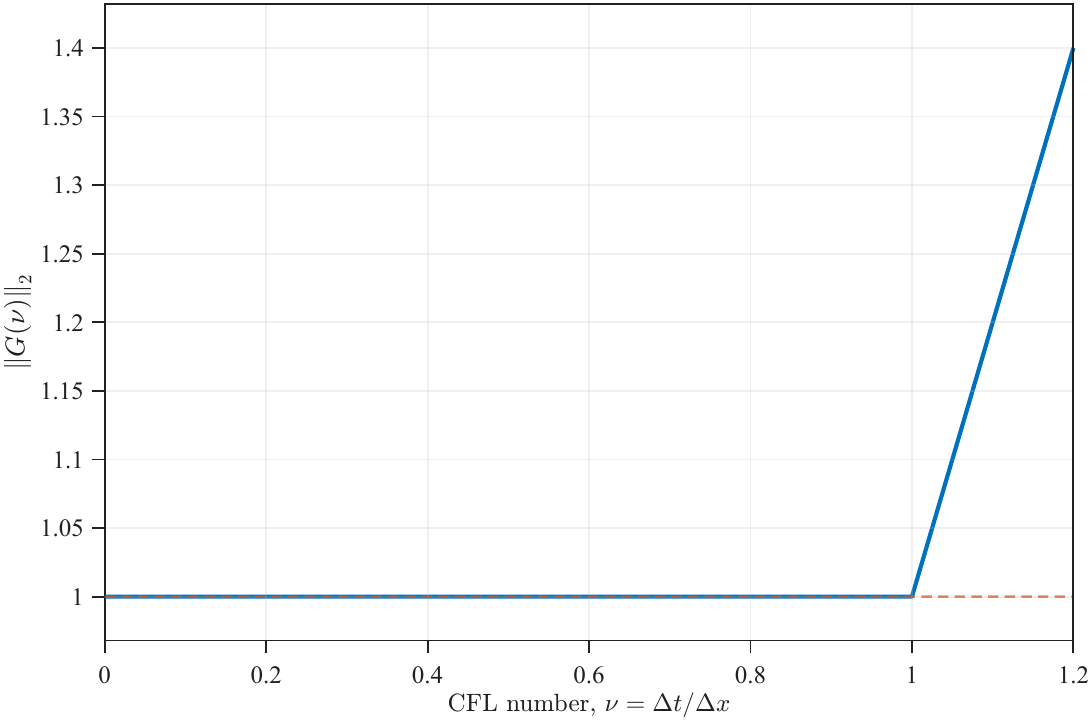}
        \caption{Isotropic velocity $\boldsymbol c=(1,1,1,1)$.}
        \label{fig:fe-norm-scan-d4-isotropic}
    \end{subfigure}
    \hfill
    \begin{subfigure}[t]{0.485\textwidth}
        \centering
        \includegraphics[width=\linewidth]{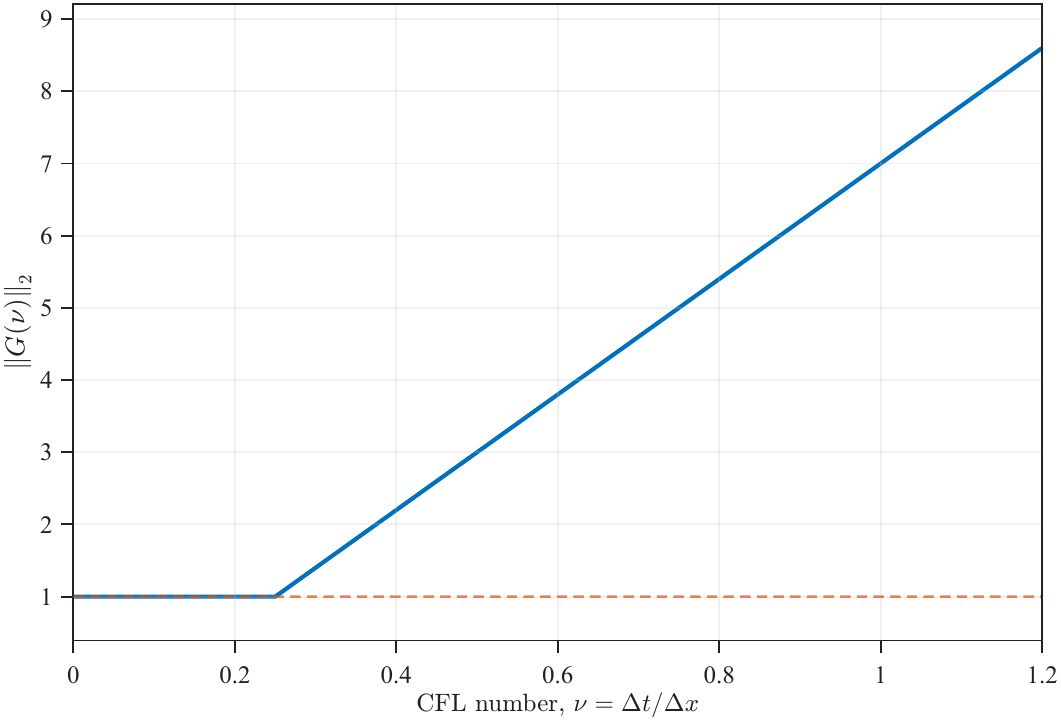}
        \caption{Anisotropic velocity $\boldsymbol c=(1,2,3,4)$.}
        \label{fig:fe-norm-scan-d4-anisotropic}
    \end{subfigure}
    \caption{Amplification-norm scans in four dimensions. The isotropic
    curve has its breakpoint at $\nu=1$, whereas the coefficient $c_4=4$
    moves the anisotropic breakpoint to $\nu=1/4$. Beyond these endpoints,
    the two curves follow the linear branch $2m\nu-1$ with $m=1$ and $m=4$,
    respectively.}
    \label{fig:fe-norm-scan-d4}
\end{figure}

To see whether the one-step norm calculation is reflected in repeated time
stepping, we evolve a nonzero state $u_h^0\in\cS_N^{(4)}$ with
\eqref{eq:d-sparse-update} and record
$E^n:=\norm{u_h^n}_{L^2}/\norm{u_h^0}_{L^2}$ at $t_n=n\Delta t$.
For equal speeds, $\nu=1$ puts all four directional Courant numbers at their
endpoint, and the normalized norm decreases throughout the computation.
Raising $\nu$ to $1.001$ leaves the initial transient nearly unchanged, but
an unstable component eventually dominates; see
Figure~\ref{fig:l2-history-d4-isotropic}.

The anisotropic run isolates the role of the fastest direction. At
$\nu=0.25$, the four directional Courant numbers are
$(0.25,0.5,0.75,1)$, so only the $x_4$ direction reaches its endpoint. A
small increase to $\nu=0.251$ changes the limiting value to $c_4\nu=1.004$
and produces the growth displayed in
Figure~\ref{fig:l2-history-d4-anisotropic}.

\begin{figure}[htbp]
    \centering
    \begin{subfigure}[t]{0.485\textwidth}
        \centering
        \includegraphics[width=\linewidth]{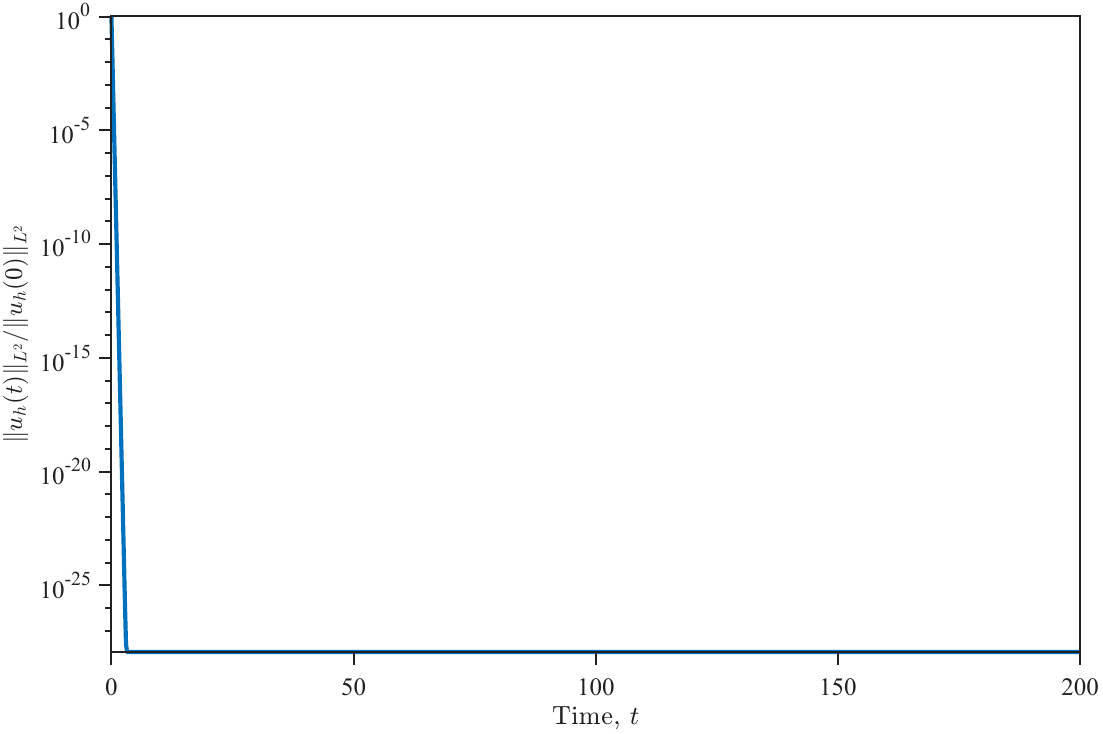}
        \caption{Endpoint run, $\nu=1$.}
        \label{fig:l2-d4-isotropic-cfl1000}
    \end{subfigure}
    \hfill
    \begin{subfigure}[t]{0.485\textwidth}
        \centering
        \includegraphics[width=\linewidth]{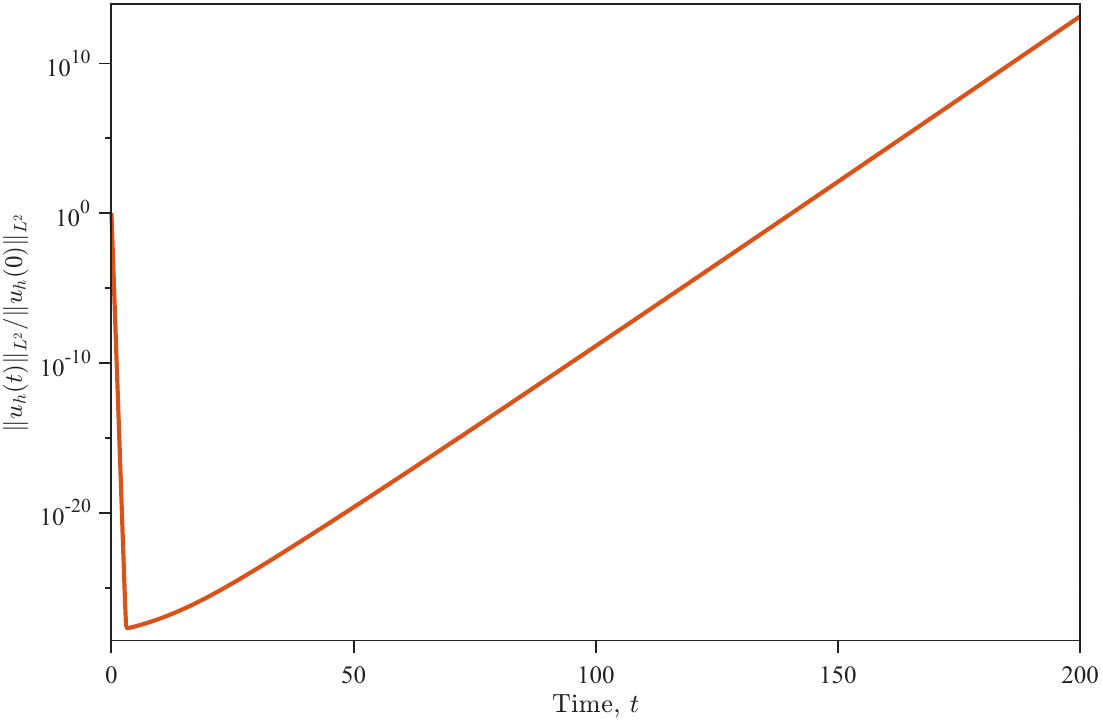}
        \caption{A $0.1\%$ CFL increase, $\nu=1.001$.}
        \label{fig:l2-d4-isotropic-cfl1001}
    \end{subfigure}
    \caption{Long-time response for four equal transport rates. At
    $\nu=1$ (left), all coordinate directions are simultaneously limiting
    and the normalized norm remains bounded. At $\nu=1.001$ (right), the
    supercritical component emerges after a short period of decay and then
    grows exponentially.}
    \label{fig:l2-history-d4-isotropic}
\end{figure}

\begin{figure}[htbp]
    \centering
    \begin{subfigure}[t]{0.485\textwidth}
        \centering
        \includegraphics[width=\linewidth]{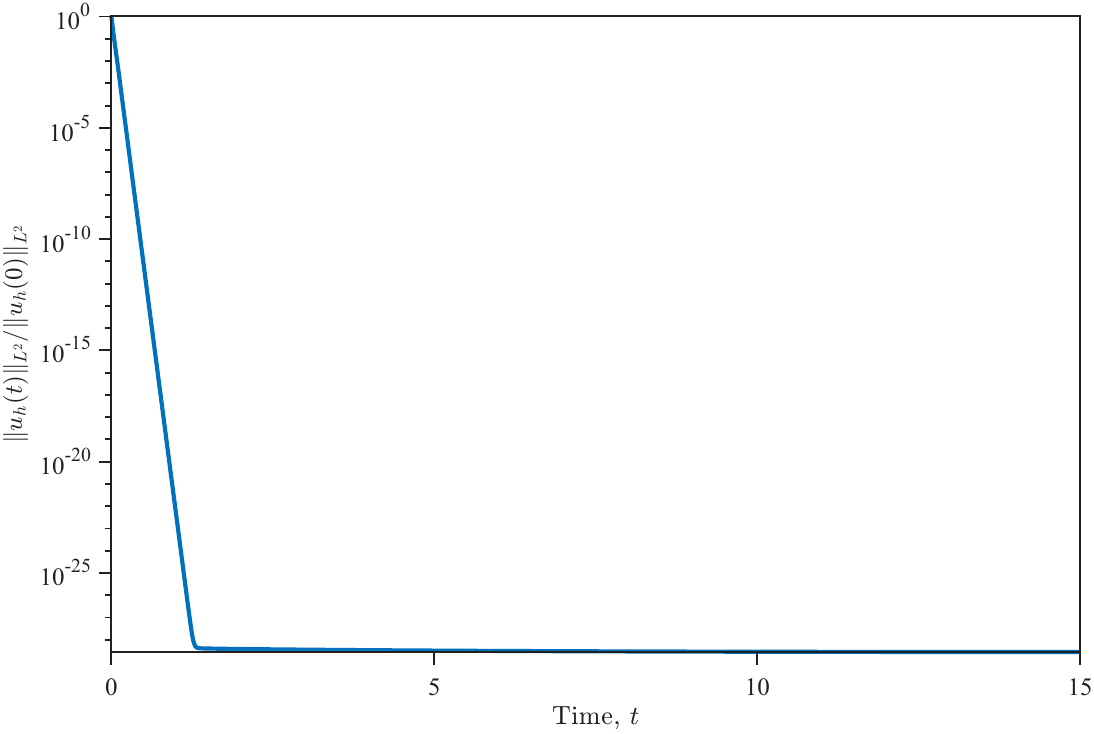}
        \caption{Endpoint fixed by $c_4$, $\nu=0.250$.}
        \label{fig:l2-d4-anisotropic-cfl0250}
    \end{subfigure}
    \hfill
    \begin{subfigure}[t]{0.485\textwidth}
        \centering
        \includegraphics[width=\linewidth]{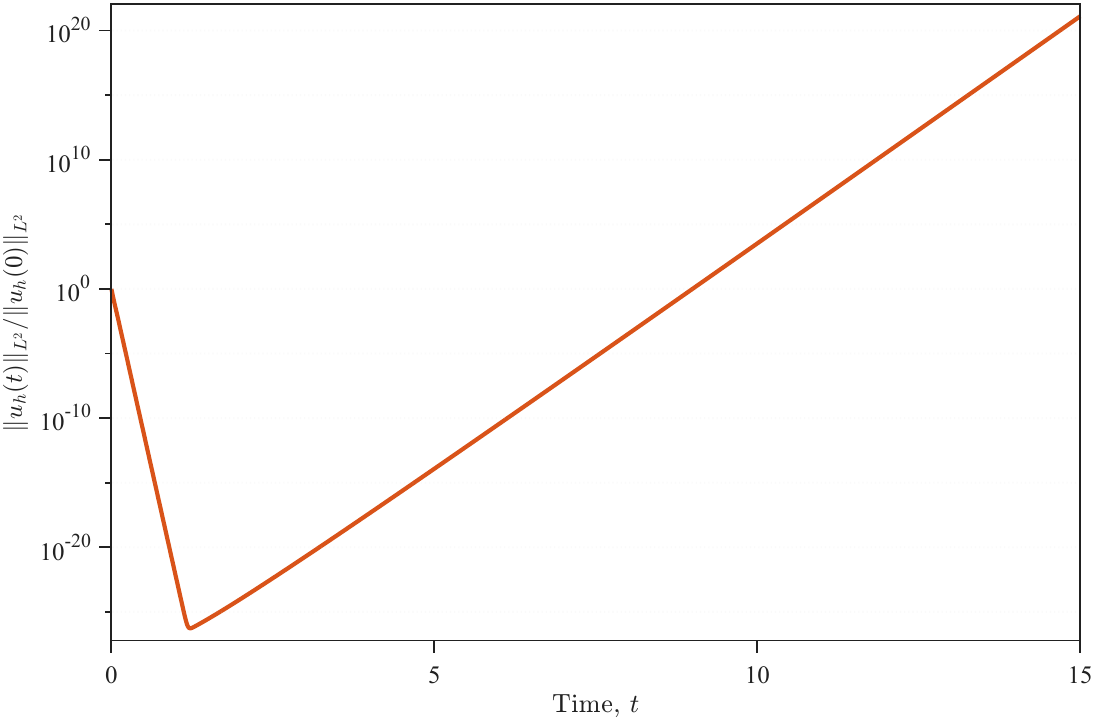}
        \caption{Slightly supercritical run, $\nu=0.251$.}
        \label{fig:l2-d4-anisotropic-cfl0251}
    \end{subfigure}
    \caption{Anisotropic four-dimensional test with
    $(c_1,c_2,c_3,c_4)=(1,2,3,4)$. The endpoint run (left) has
    $c_4\nu=1$ and remains stable. In the right panel, $c_4\nu=1.004$;
    the departure from the unit-norm regime is followed by exponential
    growth.}
    \label{fig:l2-history-d4-anisotropic}
\end{figure}

For comparison, a full grid would restrict $\nu$ to $1/4$ and $1/10$ in
these two tests. Thus the sparse grid gains are $4$ for equal speeds---the largest possible gain in four dimensions---and $5/2$ for the anisotropic
choice. The examples therefore capture both the absence of a dimensional
sum-rate penalty and the sharp loss of stability beyond the max-rate
endpoint.

\subsection{General downward closed index sets}
\label{subsec:numerics-dc}

Finally, we verify Theorems~\ref{thm:dc-sufficient} and
\ref{thm:dc-sharpness} on representative downward closed index sets in
two and three dimensions. For each pair $(\Lambda,\boldsymbol c)$, we
assemble the amplification operator
$G_\Lambda(\Delta t):=I-\sum_\ell(c_\ell\Delta t/h_\ell)
P_{\cS_\Lambda}A_\ell$ on $\cS_\Lambda$ in the orthonormal Haar tensor
basis and evaluate $\norm{G_\Lambda(\Delta t)}_2$ by a singular value
decomposition. Since $\Delta t\mapsto\norm{G_\Lambda(\Delta t)}_2$ is
convex and equals $1$ at $\Delta t=0$, the contractivity threshold
$\Delta t^*:=\sup\{\Delta t:\norm{G_\Lambda(\Delta t)}_2\leq1\}$ is
computed by bisection to relative accuracy $10^{-12}$.
Table~\ref{tab:dc-numerics} compares $\Delta t^*$ with the bound
$\Delta t_\Lambda:=1/\mathcal C_\Lambda(\boldsymbol c)$ of
Theorem~\ref{thm:dc-sufficient}. In all cases where the corner criterion
of Theorem~\ref{thm:dc-sharpness} is satisfied, the computed threshold
agrees with $\Delta t_\Lambda$ to machine precision. For the L-shaped set
of Example~\ref{ex:dc-Lshape}, whose unique maximizer has its corner
outside $\Lambda$, the bound is strictly smaller than the true threshold,
in agreement with the theory. Below, $\Lambda\langle G\rangle$ denotes
the downward closure of the generator set $G$.

\begin{table}[htbp]
    \centering
    \caption{Sufficient CFL bound $\Delta t_\Lambda$ of
    Theorem~\ref{thm:dc-sufficient} versus the numerically computed
    contractivity threshold $\Delta t^*$ for various downward closed
    index sets. The column ``criterion'' indicates whether the corner
    criterion of Theorem~\ref{thm:dc-sharpness} is satisfied.}
    \label{tab:dc-numerics}
    \begin{tabular}{cllcccc}
        \toprule
        $d$ & $\Lambda$ & $\boldsymbol c$ &
        $\Delta t_\Lambda$ & criterion & $\Delta t^*$ &
        $\Delta t^*/\Delta t_\Lambda$ \\
        \midrule
        2 & $\Lambda\langle(4,2)\rangle$ & $(1,1)$
          & $1/20$  & yes & $1/20$  & $1.0000$ \\
        2 & $\Lambda\langle(4,2)\rangle$ & $(2,5)$
          & $1/52$  & yes & $1/52$  & $1.0000$ \\
        2 & $\{2p_1+p_2\leq6\}$ & $(1,1)$
          & $1/64$  & yes & $1/64$  & $1.0000$ \\
        2 & $\{2p_1+p_2\leq6\}$ & $(2,5)$
          & $1/320$ & yes & $1/320$ & $1.0000$ \\
        2 & $\Lambda\langle(3,1),(2,2),(0,4)\rangle$ & $(1,1)$
          & $1/16$  & yes & $1/16$  & $1.0000$ \\
        2 & $\Lambda\langle(3,1),(2,2),(0,4)\rangle$ & $(2,5)$
          & $1/80$  & yes & $1/80$  & $1.0000$ \\
        2 & $\Lambda\langle(4,1),(1,3)\rangle$ (L-shaped) & $(1,1)$
          & $1/24$  & no  & $5.5359\times10^{-2}$ & $1.3286$ \\
        2 & $\Lambda\langle(4,1),(1,3)\rangle$ (L-shaped) & $(2,5)$
          & $1/72$  & no  & $2.1476\times10^{-2}$ & $1.5463$ \\
        3 & $\{|\boldsymbol p|_1\leq3\}$ & $(1,2,3)$
          & $1/24$  & yes & $1/24$  & $1.0000$ \\
        3 & $\{p_1+2p_2+3p_3\leq6\}$ & $(1,2,3)$
          & $1/64$  & yes & $1/64$  & $1.0000$ \\
        \bottomrule
    \end{tabular}
\end{table}

\section{Conclusions}\label{sec:conclusion}

In this paper, we establish the sharp forward Euler CFL condition for the piecewise constant sparse grid DG scheme discretization for transport equations in arbitrary dimensions. The proof combines exact one-dimensional Haar compression and leakage identities with a collective estimate of the mixed directional terms, while axis-aligned modes in the finest level establish sharpness.

We further generalize these results to more general hierarchical spaces built on downward closed index sets. We give an explicit sufficient CFL condition, and identify a geometric criterion under which it is sharp, covering in particular anisotropic full grids and the total-level space. The L-shaped example shows that the condition can be strictly sufficient when the criterion fails, and a complete characterization of the sharp threshold for arbitrary downward closed sets remains open. The analysis relies on the Haar wavelets structure, constant coefficients, and periodicity. Higher-order sparse grid DG methods and explicit RK time integrators are further  directions; their analysis will require new multilevel estimates for the coupling across levels and coordinate directions.

\section*{Acknowledgments}
The author thanks Yingda Cheng for first bringing this problem to his attention at Michigan State University.
This work was partially supported by NSF grant DMS-2618114, ONR grant N00014-24-1-2242, and Ralph E. Powe Junior Faculty Enhancement Award from Oak Ridge Associated Universities (ORAU).

\section*{Use of AI tools}
ChatGPT was used in the preparation of this manuscript for drafting, language editing, consistency checks, and assisting with the proof of Theorem~\ref{thm:dc-sufficient} on the sufficient CFL condition on general downward closed sets. The author is responsible for checking every statement and proof and for the correctness of the final manuscript.

\bibliographystyle{plain}
\bibliography{ref}

\end{document}